\documentclass[10pt]{article} 
\usepackage[utf8]{inputenc} 
\usepackage[english]{babel}
\usepackage[T1]{fontenc}
\usepackage{latexsym}

\usepackage{geometry} 
\usepackage{booktabs} 
\usepackage{array} 
\usepackage{paralist} 
\usepackage{verbatim} 
\usepackage{subfig} 
\usepackage{amsmath}
\usepackage{amsfonts}
\usepackage{hyperref}
\usepackage{bbm}
\usepackage{babel}
\usepackage{amsthm}
\usepackage{amssymb}
\usepackage{mathtools}
\usepackage{bbold}

\makeatletter
\newtheorem*{rep@theorem}{\rep@title}
\newcommand{\newreptheorem}[2]{%
\newenvironment{rep#1}[1]{%
 \def\rep@title{#2 \ref{##1}}%
 \begin{rep@theorem}}%
 {\end{rep@theorem}}}
\makeatother

\newtheorem{thm}{Theorem}[section]
\newtheorem{cor}[thm]{Corollary}
\newtheorem{lem}[thm]{Lemma}
\newtheorem{prop}[thm]{Proposition}
\newtheorem{defn}[thm]{Definition}
\newtheorem{rem}[thm]{Remark}

\newtheorem*{cor*}{Corollary}

\newtheorem{theorem}{Theorem}
\newreptheorem{theorem}{Theorem}

\usepackage{fancyhdr} 
\usepackage{sectsty}
\allsectionsfont{\sffamily\mdseries\upshape} 

\usepackage[nottoc,notlof,notlot]{tocbibind} 
\usepackage[titles,subfigure]{tocloft} 

\title{$L^{1}$ metric geometry of potentials with prescribed singularities on compact Kähler manifolds}
\author{Antonio Trusiani\footnote{email: \href{mailto:antonio.trusiani91@gmail.com}{antonio.trusiani91@gmail.com}}}
\date{} 

\begin{document}
\maketitle
\begin{abstract}
Given $(X,\omega)$ compact Kähler manifold and $\psi\in\mathcal{M}^{+}\subset PSH(X,\omega)$ a \emph{model type envelope} with non-zero mass, i.e. a fixed potential determining a singularity type such that $\int_{X}(\omega+dd^{c}\psi)^{n}>0$, we prove that the $\psi-$relative finite energy class $\mathcal{E}^{1}(X,\omega,\psi)$ becomes a complete metric space if endowed with a distance $d$ which generalizes the well-known $d_{1}$ distance on the space of Kähler potentials.\\
Moreover, for $\mathcal{A}\subset \mathcal{M}^{+}$ totally ordered, we equip the set $X_{\mathcal{A}}:=\bigsqcup_{\psi\in\overline{\mathcal{A}}}\mathcal{E}^{1}(X,\omega,\psi)$ with a natural distance $d_{\mathcal{A}}$ which coincides with $d$ on $\mathcal{E}^{1}(X,\omega,\psi)$ for any $\psi\in\overline{\mathcal{A}}$. We show that $\big(X_{\mathcal{A}},d_{\mathcal{A}}\big)$ is a complete metric space.\\
As a consequence, assuming $\psi_{k}\searrow \psi$ and $\psi_{k},\psi\in \mathcal{M}^{+}$, we also prove that $\big(\mathcal{E}^{1}(X,\omega,\psi_{k}),d\big)$ converges in a Gromov-Hausdorff sense to $\big(\mathcal{E}^{1}(X,\omega,\psi),d\big)$ and that there exists a direct system $\Big\langle\big(\mathcal{E}^{1}(X,\omega,\psi_{k}),d\big),P_{k,j}\Big\rangle$ in the category of metric spaces whose direct limit is dense into $\big(\mathcal{E}^{1}(X,\omega,\psi),d\big)$.
\end{abstract}
\vspace{10pt}

\vspace{10pt}

\section{Introduction}
In the last forty years it has become important to understand the space of Mabuchi $\mathcal{H}$, i.e. the space of Kähler potentials in a fixed Kähler cohomology class $\{\omega\}\in H^{2}(X,\mathbbm{R})\cap H^{1,1}(X)$ for $(X,\omega)$ compact Kähler manifold of dimension $n$:
$$
\mathcal{H}:=\{\varphi\in\mathcal{C}^{\infty}\, :\, \omega+dd^{c}\varphi \, \mathrm{is}\, \mathrm{a}\, \mathrm{Kahler} \, \mathrm{form}\},
$$
where $d^{c}:=\frac{i}{2\pi}(\partial-\bar{\partial})$, so that $dd^{c}=\frac{i}{\pi}\partial\bar{\partial}$. By the pioneering papers \cite{Mab86}, \cite{Sem92} and \cite{Don99}  $\mathcal{H}$ can be endowed with a Riemannian structure given by the metric
$$
(f,g)_{\varphi}:=\Big(\int_{X}fg(\omega+dd^{c}\varphi)^{n}\Big)^{1/2}
$$
where $\varphi\in\mathcal{H}$, $f,g\in T_{\varphi}\mathcal{H}\simeq C^{\infty}(X)$ and the metric geodesic segments are solutions of homogeneous complex Monge-Ampère equations (see also \cite{Chen00}). Later Darvas introduced in \cite{Dar14} the Finsler metric $ |f|_{1,\varphi}:=\int_{X}|f|(\omega+dd^{c}\varphi)^{n} $ on $\mathcal{H}$ with associated distance $d_{1}$, that we will simply denote by $d$. The metric completion of $(\mathcal{H},d)$ has a pluripotential description (\cite{Dar14}) since it coincides with
$$
\mathcal{E}^{1}(X,\omega):=\big\{u\in PSH(X,\omega) \, :\,  E(u)>-\infty\big\}
$$
where $E(\cdot)$ is the \emph{Aubin-Mabuchi energy} defined as
$$
E(u):=\frac{1}{n+1}\sum_{j=0}^{n}\int_{X}u\omega^{j}\wedge(\omega+dd^{c}u)^{n-j}
$$
if $u$ is locally bounded and as $E(u):=\lim_{j\to\infty}E\big(\max(u,-j)\big)$ otherwise (see \cite{Mab86}, \cite{Aub84}, \cite{BB08} and \cite{BEGZ10}). Here for the wedge product among $(1,1)$-currents we mean the \emph{non-pluripolar product} (see \cite{BEGZ10}). Moreover the $d$-distance can be expressed as
$$
d(u,v):=E(u)+E(v)-2E\big(P_{\omega}(u,v)\big),
$$
where $P_{\omega}(u,v):=\sup\{w\in PSH(X,\omega)\, :\, w\leq \min(u,v)\}$ is the \emph{rooftop envelope} operator introduced in \cite{RWN14}. The complete geodesic metric space $\big(\mathcal{E}^{1}(X,\omega),d\big)$ turned out to be very useful to formulate in analytic terms and in some cases to solve important conjectures regarding the search of special metrics (see \cite{BBGZ09}, \cite{DR15}, \cite{BBEGZ16}, \cite{BDL16}, \cite{BBJ15}, \cite{DH17}, \cite{CC17}, \cite{CC18a}, \cite{CC18b}). Furthermore the metric topology is related to the continuity of the Monge-Ampère operator since it coincides with the so-called \emph{strong topology} (\cite{BBEGZ16}).\\

The space $\mathcal{E}^{1}(X,\omega)$ contains only potentials which are at most slightly singular (see \cite{DDNL17a}). Thus Darvas, Di Nezza are Lu introduced in \cite{DDNL17b} the analogous set $\mathcal{E}^{1}(X,\omega,\psi)$ with respect to a fixed $\omega$-psh function $\psi$. More precisely,
$$
\mathcal{E}^{1}(X,\omega,\psi):=\big\{u\in PSH(X,\omega) \, : \, u\leq \psi+C \, \mathrm{for}\, C\in\mathbbm{R}\,\, \mathrm{and}\,\, E_{\psi}(u)>-\infty\big\},
$$
where
$$
E_{\psi}(u):=\frac{1}{n+1}\sum_{j=0}^{n}\int_{X}(u-\psi)(\omega+dd^{c}\psi)^{j}\wedge (\omega+dd^{c}u)^{n-j}
$$
if $|u-\psi|$ is globally bounded and $E_{\psi}(u):=\lim_{j\to \infty}E_{\psi}\big(\max(u,	\psi-j)\big)$ otherwise. One of the reasons that leads them to investigate and develop the pluripotential theory of these sets was the search of solution with prescribed singularities $[\psi]$ for the complex Monge-Ampère equation $(\omega+dd^{c}u)^{n}=\mu$ (see also \cite{DDNL18b}). They found out that there is a necessary condition to assume on $\psi$: $\psi-P_{\omega}[\psi](0)$ must be globally bounded where $P_{\omega}[\psi](0):=\big(\lim_{C\to \infty}P_{\omega}(\psi_{j}+C,0)\big)^{*}$, (\cite{RWN14}, the star is for the upper semicontinuous regularization). So, without loss of generality, one may assume that $\psi$ is a \emph{model type envelope}, i.e. $\psi=P_{\omega}[\psi](0)$ (see section \ref{sec:Preli}). In this setting they were able to show the existence of Kähler-Einstein metric with prescribed singularities $[\psi]$ in the case of $X$ manifold with ample canonical bundle and in the case $X$ Calabi-Yau manifold.\\

Therefore one of the main motivations for this paper is to endow the set $\mathcal{E}^{1}(X,\omega,\psi)$ with a metric structure to address in a future work the problem of characterizing analytically the existence of Kähler-Einstein metrics with prescribed singularities in the Fano case.\\
Thus, assuming $\psi$ to be a model type envelope and defining
$$
d(u,v):=E_{\psi}(u)+E_{\psi}(v)-2E_{\psi}\big(P_{\omega}(u,v)\big)
$$
on $\mathcal{E}^{1}(X,\omega,\psi)\times\mathcal{E}^{1}(X,\omega,\psi)$, we prove the following theorem.
\begin{theorem}\footnote{The assumption on $\omega$ to be Kähler is unnecessary, i.e. this Theorem easily extends to the big case.}
\label{thmA}
Let $\psi\in PSH(X,\omega)$ be a model type envelope with non-zero mass $V_{\psi}=\int_{X}(\omega+dd^{c}\psi)^{n}>0$. Then $\big(\mathcal{E}^{1}(X,\omega,\psi),d\big)$ is a complete metric space.
\end{theorem}
The non-zero total mass $V_{\psi}>0$ condition is a necessary hypothesis because otherwise $d\equiv 0$ (Remark \ref{rem:ZeroMass}). \\

The second main motivation of the paper is to set up a new way to compare solutions of the complex Monge-Ampère equation $(\omega+dd^{c}u)^{n}=\mu$ on $\mathcal{E}^{1}(X,\omega,\psi)$ varying the model type envelope $\psi$ (see \cite{Tru20}). This leads to wonder, first of all, \emph{how} a sequence of spaces $\mathcal{E}^{1}(X,\omega,\psi_{k})$ \emph{converges} to $\mathcal{E}^{1}(X,\omega,\psi)$ if $\psi_{k}\to \psi$. The most interesting case seems to be when $\{\psi_{k}\}_{k\in\mathbbm{N}}$ is \emph{totally ordered} with respect to the natural partial order $\preccurlyeq$ on $PSH(X,\omega)$ given by $u\preccurlyeq v$ if $u\leq v+C$ for a constant $C\in\mathbbm{R}$.\\
Thus in the second part of the paper, denoting with $\mathcal{M}$ the set of all model type envelopes and with $\mathcal{M}^{+}$ its elements with non-zero mass, we assume to have a totally ordered subset $\mathcal{A}\subset \mathcal{M}^{+}$ and we define
$$
X_{\mathcal{A}}:=\bigsqcup_{\psi\in\mathcal{\overline{A}}}\mathcal{E}^{1}(X,\omega,\psi)
$$
where $\mathcal{\overline{A}}\subset \mathcal{M}$ is the closure of the set $\mathcal{A}$ as subset of $PSH(X,\omega)$ with its $L^{1}$-topology. Our next result regards the existence of a natural metric topology on $X_{\mathcal{A}}$ induced by a distance $d_{\mathcal{A}}$ which extends the distance $d$ over $\mathcal{E}^{1}(X,\omega,\psi)$ for any $\psi\in\overline{\mathcal{A}}$ (see section \ref{sec:5}). The construction of the metric $d_{\mathcal{A}}$ is quite technical, but the idea is to exploit the compactness of \textit{potentials with uniformly bounded entropy} and the interesting Proposition \ref{prop:Dist} to make a natural definition of a semimetric $\tilde{d}_{\mathcal{A}}$ (see Definition \ref{defn:Ref} and Proposition \ref{prop:Proper}) from which the metric $d_{\mathcal{A}}$ in obtained in a classical way (see page $17$). \newline
\begin{theorem}
\label{thmC}
Let $\mathcal{A}\subset \mathcal{M}^{+}$ totally ordered. Then $(X_{\mathcal{A}},d_{\mathcal{A}})$ is a complete metric space and $d_{\mathcal{A}}$ restricts to $d$ on $\mathcal{E}^{1}(X,\omega,\psi)\times \mathcal{E}^{1}(X,\omega,\psi)$ for any $\psi\in\mathcal{\overline{A}}$.
\end{theorem}
Here for $\psi\in \mathcal{M}\setminus \mathcal{M}^{+}$ we identify the set $\mathcal{E}^{1}(X,\omega,\psi)$ with a singleton $P_{\psi}$.\newline
Let me stress further that the distance $d_{\mathcal{A}}$ is a natural generalization of the distances $d$: in the companion paper \cite{Tru20} we show how its metric topology defines a strong topology which is heavily related to the continuity of the Monge-Ampère operator.\\

As a consequence of Theorem \ref{thmC}, considering a decreasing sequence $\{\psi_{k}\}_{k\in\mathbbm{N}}\subset \mathcal{M}^{+}$ converging to $\psi\in\mathcal{M}^{+}$, one immediately thinks that the metric spaces $\big(\mathcal{E}^{1}(X,\omega,\psi_{k}),d\big)$ essentially \emph{converges} to $\big(\mathcal{E}^{1}(X,\omega,\psi),d\big)$. The problem here is that these metric spaces are not locally compact, therefore it is not clear what kind of convergence one should look at. In section \ref{sec:5} we introduce the \emph{compact pointed Gromov-Hausdorff convergence} ($cp$-$GH$) which basically mimic the pointed Gromov-Hausdorff convergence (see \cite{BH99} and \cite{BBI01}) replacing, for any space, the family of balls centered at the point with an increasing family with dense union of compact sets containing the point chosen (see Definition \ref{defn:cp-GH}). 
\begin{theorem}
\label{thmD}
Let $\{\psi_{k}\}_{k\in\mathbbm{N}}\subset \mathcal{M}^{+}$ be a decreasing sequence converging to $\psi\in\mathcal{M}^{+}$. Then
$$
\Big(\mathcal{E}^{1}(X,\omega,\psi_{k}),\psi_{k},d\Big)\xrightarrow{cp-GH} \Big(\mathcal{E}^{1}(X,\omega,\psi),\psi,d\Big).
$$
\end{theorem}

Furthermore we show that the maps
$$
P_{i,j}:=P_{\omega}[\psi_{j}](\cdot):\big(\mathcal{E}^{1}(X,\omega,\psi_{i},d)\big)\to \big(\mathcal{E}^{1}(X,\omega,\psi_{j}),d\big)
$$
for $i\leq j$ are \emph{short maps} (i.e. $1$-Lipschitz). Hence $\big\langle \big(\mathcal{E}^{1}(X,\omega,\psi_{i}),d\big), P_{i,j}\big\rangle$ is a direct system in the category of metric spaces. We denote with $\mathfrak{m}-\lim_{\longrightarrow}$ the direct limit in this category.
\begin{theorem}
\label{thmE}
Let $\{\psi_{k}\}_{k\in\mathbbm{N}}\subset \mathcal{M}^{+}$ be a decreasing sequence converging to $\psi\in\mathcal{M}^{+}$. Then there is an isometric embedding
$$
\mathfrak{m}-\lim_{\longrightarrow}\big\langle\big(\mathcal{E}^{1}(X,\omega,\psi_{i}),d\big), P_{i,j}\big\rangle\hookrightarrow \big(\mathcal{E}^{1}(X,\omega,\psi),d\big)
$$
with dense image equal to $\bigcup_{k\in\mathbbm{N}}P_{\omega}[\psi]\big(\mathcal{E}^{1}(X,\omega,\psi_{k})\big)$.
\end{theorem}
\subsection{Related Works}
During the last period of the preparation of this paper, Xia in \cite{X19b} independently showed Theorem \ref{thmA} as particular case of his main Theorem.  
\subsection{Structure of the paper}
After recalling some preliminaries in section \ref{sec:Preli}, the third section is dedicated to prove Theorem \ref{thmA}. In this section many of the proofs are just easily adapted to our setting from the absolute setting in the Kähler and in the big case (in particular \cite{DDNL18a}).\\
Section \ref{sec:5} is the core of the paper, where we show Theorems \ref{thmC}, \ref{thmD}, and \ref{thmE}.
\subsection{Acknowledgments}
The author is grateful to his two advisors S. Trapani and D. Witt Nyström for their comments and suggestions. He would also like to thank M. Xia for inspiring talks.
\section{Preliminaries}
\label{sec:Preli}
Let $(X,\omega)$ be a compact Kähler manifold ($\omega$ fixed Kähler form on $X$). We denote with $PSH(X,\omega)$ the set of all $\omega$-psh ($\omega-$plurisubharmonic) functions on $X$, i.e. the set of all functions $u$ given locally as sum of a plurisubharmonic function and a smooth function such that $\omega+dd^{c}u\geq 0$ as $(1,1)$-current. Here $d^{c}:=\frac{i}{2\pi}(\partial-\bar{\partial})$ so that $dd^{c}=\frac{i}{\pi}\partial\bar{\partial}$. We say that $u$ is more singular than $v$ if there exists a constant $C\in\mathbbm{R}$ such that $u\leq v+C$. Being more/less singular is a partial order on $PSH(X,\omega)$. We use $\preccurlyeq$ to denote such order, and we indicate with $[u]$ the class of equivalence with respect to this order. Moreover, according to the notation in $\cite{DDNL17b}$, $PSH(X,\omega,\psi)$ is the set of all $u\in PSH(X,\omega)$ such that $u \preccurlyeq \psi$, and $u\in PSH(X,\omega,\psi)$ is said to have \emph{$\psi$-relative minimal singularities} if $u\in [\psi]$. To start the investigation of these functions we recall the construction of the envelopes introduced in \cite{RWN14}: for any couple of $\omega-$psh functions $u,v$, the function
$$
P_{\omega}[u](v):=\Big(\lim_{C\to +\infty}P_{\omega}(u+C,v)\Big)^{*}
$$
is $\omega-$psh, where $P_{\omega}(u,v):=\sup\{w\in PSH(X,\omega) \, :\, w\leq \min(u,v)\}$ is the \emph{rooftop envelope} (the star is for the upper semicontinuous regularization). Roughly speaking if $u\preccurlyeq v$ then $P_{\omega}[u](v)$ is the largest $\omega-$psh function that is bounded from above by $v$ and that preserves the singularities type $[u]$. We say that an $\omega$-psh function $\psi$ is a \emph{model type envelope} if $P_{\omega}[\psi]:=P_{\omega}[\psi](0)=\psi$. There are plenty of these functions and $P_{\omega}[P_{\omega}[\psi]]=P_{\omega}[\psi]$ if $\int_{X}\big(\omega+dd^{c}\psi)^{n}>0$ by Theorem $3.12$ in \cite{DDNL17b} (in the sense of non-pluripolar product as explained below). Hence $\psi\to P_{\omega}[\psi]$ may be thought as a projection from the set of $\omega$-psh functions to the set of model type envelopes when restricted to functions with positive Monge-Ampère mass. We refer to Remark $1.6$ in \cite{DDNL17b} for some tangible examples of these functions. Denoting with $\mathcal{M}$ the set of model type envelopes, it is easy to see that if $\psi_{1},\psi_{2}\in PSH(X,\omega)$ satisfy $\psi_{1}\preccurlyeq \psi_{2}$ then $\psi_{1}\leq \psi_{2}$. Hence the partial orders $\leq, \preccurlyeq$ coincide on $\mathcal{M}$.\\

Given $T_{1},\cdots,T_{p}$ closed and positive $(1,1)$-currents, with $T_{1}\wedge \cdots \wedge T_{p}$ we will mean the \emph{non-pluripolar} product (see \cite{BEGZ10}). It is always well-defined on a compact Kähler manifold (Proposition $1.6$ in \cite{BEGZ10}) and it is \emph{local in the plurifine topology}, i.e. in the coarsest topology with respect to which all psh functions on all open subsets of $X$ become continuous (see also \cite{BT87}). Moreover, setting $\omega_{\varphi}:=\omega+dd^{c}\varphi$ if $\varphi\in PSH(X,\omega)$, the map
$$
PSH(X,\omega)\ni \varphi\to V_{\varphi}:=\int_{X}\omega_{\varphi}^{n}\in\mathbbm{R}
$$
respects the partial order defined before by the main theorem in \cite{WN17}, i.e. if $u \preccurlyeq v$ then $V_{u}\leq V_{v}$. Such monotonicity still holds considering the mixed product, i.e. $\int_{X}\omega_{u_{1}}\wedge\cdots \wedge\omega_{u_{n}}\leq \int_{X}\omega_{v_{1}}\wedge\cdots\wedge \omega_{v_{n}}$ if $u_{j}\preccurlyeq v_{j}$ for any $j=1,\dots,n$ (Theorem $2.4$ in \cite{DDNL17b}). Generally we have the following principle:
\begin{prop}\emph{(Comparison Principle).}
\label{prop:Comp}
Let $u,v\in PSH(X,\omega)$ such that $u \preccurlyeq v$, and $w_{1},\dots,w_{n-p}\in PSH(X,\omega)$ for $1\leq p\leq n$ integer. Then
$$
\int_{\{v<u\}}\omega_{u}^{p}\wedge\omega_{w_{1}}\wedge\cdots\wedge\omega_{w_{n-p}}\leq\int_{\{v<u\}}\omega_{v}^{p}\wedge\omega_{w_{1}}\wedge\cdots\wedge\omega_{w_{n-p}}.
$$
\end{prop}
\begin{proof}
The proof proceeds as that of Corollary $1.4$ in \cite{WN17}. For any $\epsilon>0$, set $v_{\epsilon}:=\max(v,u-\epsilon)$. Thus
\begin{multline*}
\int_{X}\omega_{v}^{p}\wedge\omega_{w_{1}}\wedge\cdots\wedge\omega_{w_{n-p}}=\int_{X}\omega_{v_{\epsilon}}^{p}\wedge\omega_{w_{1}}\wedge\cdots\wedge\omega_{w_{n-p}}\geq\\
\geq\int_{\{v>u-\epsilon\}}\omega_{v}^{p}\wedge\omega_{w_{1}}\wedge\cdots\wedge\omega_{w_{n-p}}+\int_{\{v<u-\epsilon\}}\omega_{u}^{p}\wedge\omega_{w_{1}}\wedge\cdots\wedge\omega_{w_{n-p}},
\end{multline*}
which implies
$$
\int_{\{v<u-\epsilon\}}\omega_{u}^{p}\wedge\omega_{w_{1}}\wedge\cdots\wedge\omega_{w_{n-p}}\leq \int_{\{v<u\}}\omega_{v}^{p}\wedge\omega_{w_{1}}\wedge\cdots\wedge\omega_{w_{n-p}}.
$$
The result follows from letting $\epsilon\to 0$.
\end{proof}
We also recall some results of $\cite{DDNL17b}$ which will be very useful in the sequel:
\begin{lem}[Lemma $3.7$, \cite{DDNL17b}]
\label{lem:PMass}
Let $u,v\in PSH(X,\omega)$. If $P_{\omega}(u,v)\neq -\infty$, then
$$
\omega_{P_{\omega}(u,v)}^{n}\leq \mathbbm{1}_{\{P_{\omega}(u,v)=u\}}\omega_{u}^{n}+\mathbbm{1}_{\{P_{\omega}(u,v)=v\}}\omega_{v}^{n}.
$$
\end{lem}
\begin{thm}[Theorem $3.8$, \cite{DDNL17b}]
\label{thm:3.8}
Let $u,\psi\in PSH(X,\omega)$ such that $u$ is less singular than $\psi$. Then
$$
MA_{\omega}\big(P_{\omega}[\psi](u)\big)\leq \mathbbm{1}_{\{P_{\omega}[\psi](u)=u\}}MA_{\omega}(u).
$$
In particular if $\psi$ is a model type envelope then $MA_{\omega}(\psi)\leq \mathbbm{1}_{\{\psi=0\}}MA_{\omega}(0)$.
\end{thm}
\begin{thm}[Theorem 2.3, \cite{DDNL17b}]
\label{thm:conv}
Let $\{u_{j},u_{j}^{k}\}_{j=1,\dots,n}\in PSH(X,\omega)$ such that $u_{j}^{k}\to u_{j}$ in capacity as $k\to \infty$ for $j=1,\dots,n$. Then for all bounded quasi-continuous function $\chi$,
$$
\liminf_{k\to \infty}\int_{X}\chi \omega_{u_{1}^{k}}\wedge\cdots\wedge \omega_{u_{n}^{k}}\geq\int_{X}\chi \omega_{u_{1}}\wedge\cdots\wedge \omega_{u_{n}}.
$$
If additionally,
$$
\int_{X}\omega_{u_{1}}\wedge\cdots\wedge\omega_{u_{n}}\geq \limsup_{k\to \infty}\int_{X}\omega_{u_{1}^{k}}\wedge\cdots\wedge \omega_{u_{n}^{k}}
$$
then $\omega_{u_{1}^{k}}\wedge\cdots\wedge\omega_{u_{n}^{k}}\to \omega_{u_{1}}\wedge \cdots\wedge \omega_{u_{n}}$ in the weak sense of measures on $X$.
\end{thm}
We recall that a sequence $\{u_{k}\}_{k\in\mathbbm{N}}\subset PSH(X,\omega)$ converges in capacity to $u\in PSH(X,\omega)$ if for any $\delta>0$
$$
\mathrm{Cap}\big(\{|u_{k}-u|\geq \delta\}\big)\to 0
$$
where for any $B\subset X$ Borel set
$$
\mathrm{Cap}(B):=\sup\Big\{\int_{B}MA_{\omega}(u)\, : \, u\in PSH(X,\omega),\, -1\leq u\leq 0\Big\}
$$
It is also useful to recall that if $PSH(X,\omega)\ni u_{j}\searrow u\in PSH(X,\omega)$ decreasing, then $u_{j}\to u$ in capacity, and that the convergence in capacity implies the $L^{1}$-convergence (see \cite{GZ17} and reference therein).
\subsection{Potentials with $\psi$-relative full mass.}
If $u,v$ belongs to the same class $[\psi]$ then $V_{u}=V_{v}$, but there are also examples of $\omega$-psh functions $u,v$ such that $u\prec v$ and $V_{u}=V_{v}$. Thus $u\in PSH(X,\omega,\psi)$ is said to have \emph{$\psi$-relative full mass} if $V_{u}=V_{\psi}$, and the set of all $\omega$-psh functions with $\psi$-relative full mass is denoted with $\mathcal{E}(X,\omega,\psi)$ (see \cite{DDNL17b}).
\begin{thm}\emph{(Theorem 1.3, \cite{DDNL17b}).}
\label{thm:samesing}
Suppose $\psi\in PSH(X,\omega)$ such that $V_{\psi}>0$, and $u\in PSH(X,\omega,\psi)$. The following are equivalent:
\begin{itemize}
\item[(i)] $u\in \mathcal{E}(X,\omega,\psi)$;
\item[(ii)] $P_{\omega}[u](\psi)=\psi$;
\item[(iii)] $P_{\omega}[u]=P_{\omega}[\psi]$.
\end{itemize}
\end{thm}
This result suggests that any function in the class $\mathcal{E}(X,\omega,\psi)$ is at most mildly more singular than $\psi$. Moreover this also implies that $\mathcal{E}(X,\omega,\psi_{1})\cap \mathcal{E}(X,\omega,\psi_{2})=\emptyset$ if $\psi_{1},\psi_{2}$ are two different model type envelopes with non zero total masses $V_{\psi_{1}}>0$, $V_{\psi_{2}}>0$.\\

For any $u_{1},\dots,u_{p}\in PSH(X,\omega)$, and for any $j_{1},\dots,j_{p}\in\mathbbm{N}$ such that $j_{1}+\cdots+j_{p}=n$ we introduce the notation
$$
MA_{\omega}(u_{1}^{j_{1}},\dots,u_{p}^{j_{p}}):=\omega_{u_{1}}^{j_{1}}\wedge\cdots\wedge \omega_{u_{p}}^{j_{p}}
$$
for the \emph{(mixed) non-pluripolar complex Monge-Amp\'ere measure} associated (it is a positive Borel measure) and we set $MA_{\omega}(u):=MA_{\omega}(u^{n})$. Note that if $V_{\psi}>0$ then the map
$$
\mathcal{E}(X,\omega,\psi)\ni u\to MA_{\omega}(u)/V_{\psi}
$$
has image contained in the set of non-pluripolar probability measures $\mathcal{M}(X)$. Moreover if $\psi$ is also a model type envelope then this map is surjective and it descends to a bijection on the space of all closed and positive $(1,1)$-currents with $\psi$-relative full mass, i.e. on $\mathcal{E}(X,\omega,\psi)/\mathbbm{R}$ (see Theorem $A$ in \cite{BEGZ10} when $\psi=0$, Theorem $4.28$ in \cite{DDNL17b} when $\psi$ has small unbounded locus, and Theorem $4.7$ in \cite{DDNL18b} for the general case). See the companion paper \cite{Tru20} and references therein for a further analysis of the Monge-Ampère operator.\newline

The following results will be essential in the sequel.
\begin{lem}
\label{lem:Volumes1}
Let $\{\psi_{k}\}_{k\in\mathbbm{N}}\subset \mathcal{M}^{+}$ be a monotone sequence converging a.e. to $\psi\in PSH(X,\omega)$. Then $\psi\in\mathcal{M}$ and $MA_{\omega}(\psi_{k})\to MA_{\omega}(\psi)$ weakly.
\end{lem}
\begin{proof}
Assume first $\psi_{k}\searrow \psi$. \\
As $\sup_{X}\psi=0$ we clearly have $\psi\leq P_{\omega}[\psi]=:\tilde{\psi}$ which implies $\psi=P_{\omega}[\psi]$ since $\tilde{\psi}\leq P_{\omega}[\psi_{k}]=\psi_{k}$ for any $k\in\mathbbm{N}$. For the second statement, we first observe that
\begin{equation}
\label{eqn:1Maj}
\int_{X}MA_{\omega}(\psi)\geq\limsup_{k\to \infty}\int_{\{\psi\geq -C\}}MA_{\omega}(\psi_{k})
\end{equation}
for any $C\in\mathbbm{R}$ fixed. Indeed, letting $u_{k}:=\max(\psi_{k},-C-1)$, we have $\int_{\{\psi\geq -C\}}MA_{\omega}(\psi_{k})=\int_{\{\psi\geq -C\}}MA_{\omega}(u_{k})$ as $\{\psi\geq -C\}\subset \{\psi_{k}> -C-1\}$ and $\mathbbm{1}_{\{\psi_{k}>-C-1\}}MA_{\omega}(\psi_{k})=\mathbbm{1}_{\{\psi_{k}>-C-1\}}MA_{\omega}(u_{k})$ by the locality in the plurifine topology of the non-pluripolar product (\cite{BEGZ10}). Similarly $\int_{\{\psi\geq -C\}}MA_{\omega}(u)=\int_{\{\psi\geq -C\}}MA_{\omega}(\psi)$ where $u:=\max(\psi,-C-1)$. Therefore (\ref{eqn:1Maj}) is an easy consequence of Corollary $3.3.b$ in \cite{BT87} (see also \cite{BT82}) as the latter gives
$$
\limsup_{k\to \infty}\int_{\{\psi\geq -C\}}MA_{\omega}(u_{k})\leq \int_{\{\psi\geq -C\}}MA_{\omega}(u)
$$
observing that $u_{k}\searrow u$, $|u_{k}|\leq C+1$ and that $\{\psi\geq -C\}$ is a plurifine closed set.\\
On the other hand, by Theorem \ref{thm:3.8} $MA_{\omega}(\psi_{k})\leq \mathbbm{1}_{\{\psi_{k}=0\}}MA_{\omega}(0)$ for any $k\in\mathbbm{N}$. Thus for any $C\geq 0$
\begin{equation}
\label{eqn:2Maj}
\limsup_{k\to \infty}\int_{\{\psi<-C\}}MA_{\omega}(\psi_{k})\leq \limsup_{k\to\infty}\int_{\{\psi<-C\}\cap \{\psi_{k}=0\}}MA_{\omega}(0)=0,
\end{equation}
where the last equality follows from $\bigcap_{k\in\mathbbm{N}}\{\psi_{k}=0\}=\{\psi=0\}$ since $\psi_{k}\searrow \psi$ and $\psi,\psi_{k}\leq 0$. Hence combining (\ref{eqn:1Maj}) and (\ref{eqn:2Maj}) we obtain
$$
\int_{X}MA_{\omega}(\psi)\geq \limsup_{k\to \infty}\int_{X}MA_{\omega}(\psi_{k}),
$$
and Theorem \ref{thm:conv} implies that $MA_{\omega}(\psi_{k})\to MA_{\omega}(\psi)$ weakly.\\
Assume now $\psi_{k}\nearrow \psi$ almost everywhere.\\
Again by Theorem \ref{thm:conv} we immediately get $MA_{\omega}(\psi_{k})\to MA_{\omega}(\psi)$ weakly since $\psi_{k}\to\psi$ in capacity. Thus to conclude the proof it remains to prove that $\psi\geq P_{\omega}[\psi]$ as $\psi\leq P_{\omega}[\psi]$. \\
We note that $MA_{\omega}(\psi)\leq \mathbbm{1}_{\{\psi=0\}} MA_{\omega}(0)$ since $MA_{\omega}(\psi_{k})\leq\mathbbm{1}_{\{\psi_{k}=0\}}MA_{\omega}(0)$ for any $k\in\mathbbm{N}$ (Theorem \ref{thm:3.8}). Therefore
\begin{gather*}
0\leq \int_{X}\big(P_{\omega}[\psi]-\psi\big)MA_{\omega}(\psi)\leq \int_{\{\psi=0\}}(P_{\omega}[\psi]-\psi) MA_{\omega}(0)=0
\end{gather*}
where the last equality follows from $\psi\leq P_{\omega}[\psi]\leq 0$. Hence, as by hypothesis $\int_{X}MA_{\omega}(\psi)>0$, the domination principle (Proposition $3.11$ in \cite{DDNL17b}) gives $P_{\omega}[\psi]\leq \psi$, i.e. $\psi\in\mathcal{M}$.
\end{proof}
As a consequence of Lemma \ref{lem:Volumes1} we get that $\mathcal{\overline{A}}\subset\mathcal{M}$. Indeed since $\mathcal{A}$ is totally ordered, any Cauchy sequence $\{\psi_{k}\}_{k\in\mathbbm{N}}$ admits a subsequence monotonically converging a.e. to $\big(\lim_{k\to \infty}\psi_{k}\big)^{*}$. We also note that $\mathcal{\overline{A}}$ remains totally ordered.
\begin{prop}
\label{lem:Key}
Let $\{\psi_{k}\}_{k\in\mathbbm{N}}\subset \mathcal{M}^{+}$ totally ordered such that $\psi_{k}\to \psi\in\mathcal{M}$ monotonically and let $\{u_{1,k},u_{2,k}\}_{k\in\mathbbm{N}}$ be two sequences, $u_{1,k},u_{2,k}\in\mathcal{E}(X,\omega,\psi_{k})$ for any $k\in\mathbbm{N}$, converging in capacity respectively to $u_{1},u_{2}\in \mathcal{E}(X,\omega,\psi)$. Then for any $j=0,\dots,n$,
$$
MA_{\omega}(u_{1,k}^{j},u_{2,k}^{n-j})\to MA_{\omega}(u_{1}^{j},u_{2}^{n-j})
$$
weakly. Moreover, letting $\{f_{k}\}_{k\in\mathbbm{N}},f$ be uniformly bounded quasi-continuous functions such that $f_{k}\to f$ in capacity and assuming that there exists a uniform constant $C$ such that $u_{1,k},u_{2,k}\geq \psi_{k}-C$ for any $k\in\mathbbm{N}$, the weak convergence
\begin{equation}
\label{eqn:Label}
f_{k} MA_{\omega}(u_{1,k}^{j},u_{2,k}^{n-j})\to f MA_{\omega}(u_{1}^{j},u_{2}^{n-j})
\end{equation}
holds for any $j=0,\dots,n$.
\end{prop}
In the case of small unbounded loci, the proof of the weak convergence (\ref{eqn:Label}) in Proposition \ref{lem:Key} is an adaptation of the proof of Lemma $4.1$ in \cite{DDNL17b} (and it is a particular case of Theorem $2.2.$ in \cite{X19a}). However the general case is more complicated and the different proof proposed below is based on the following Lemma. We recall that the \emph{relative Monge-Ampère capacity} for $\psi\in\mathcal{M}^{+}$ is defined as
$$
\mathrm{Cap}_{\psi}(B):=\sup\Big\{\int_{B}MA_{\omega}(u)\, : \, u\in PSH(X,\omega), \psi-1\leq u\leq\psi\Big\}
$$
for any Borel set $B\subset X$ (see \cite{DDNL17b}).
\begin{lem}
\label{lem:AfterReferee}
Let $\mathcal{A}\subset \mathcal{M}^{+}$ a totally ordered family such that the model type envelope $\psi_{\min}:=\inf_{\psi\in\mathcal{A}}\psi$ belongs to $\mathcal{M}^{+}$. Then there exists a uniform constant $A>0$ such that
\begin{equation}
\label{eqn:Totally}
\mathrm{Cap}_{\psi}(B)\leq A\, \mathrm{Cap}(B)^{1/n^{2}}
\end{equation}
for any Borel set $B\subset X$ and for any $\psi\in\mathcal{M}$.
\end{lem}
Note that, as the partial order $\preccurlyeq$ coincides with $\leq $ on $\mathcal{M}^{+}$, $\inf_{\mathcal{A}}\psi\in \mathcal{M}$ but a priori it can have zero total Monge-Ampère mass.
\begin{proof}
This Lemma essentially follows from Lemma $3.2$ in \cite{Lu20}. However, as we need to keep track of the constants, we present the details of the proof as a courtesy to the reader.\newline
As the (relative) Monge-Ampère capacities are inner regular (Lemma $4.2$ in \cite{DDNL17b}, \cite{GZ17}) and any non-pluripolar set has capacity $0$, it is enough to prove (\ref{eqn:Totally}) for a general non-pluripolar compact set $B=K$. We recall the definition of the \emph{relative extremal function}
$$
V_{K,\psi}:=\sup\{u\in PSH(X,\omega)\, : \, u\preccurlyeq\psi, u\leq \psi \, \mathrm{on} \, K\}
$$
and we define $M_{K,\psi}:=\sup_{X}V_{K,\psi}^{*}$ where as usual the star is for the upper semicontinuous regularization.\newline
If $M_{K,\psi_{\min}}\geq 1$ then combining Lemma $3.2$ in \cite{Lu20}, Lemma $3.9$ in \cite{DDNL18b} and the trivial inequality $M_{K,\psi}\geq M_{K,\psi_{\min}}$ we obtain
$$
\mathrm{Cap}_{\psi}(K)\leq \frac{C}{M_{K,\psi}}\leq \frac{C}{M_{K,\psi_{\min}}}\leq \frac{C}{V_{\psi_{\min}}^{1/n}}\mathrm{Cap}_{\psi_{\min}}(K)^{1/n}
$$
for a uniform constant $C>0$ independent of $\psi\in \mathcal{A}$ and of $K\subset X$. \newline
Instead if $M_{K,\psi_{\min}}<1$, $V_{K,\psi_{\min}}^{*}-1$ is a $\omega$-psh candidate for the $\psi_{\min}$-relative Monge-Ampère capacity, hence
$$
V_{\psi_{\min}}=\int_{X}MA_{\omega}(V_{K,\psi_{\min}}^{*})=\int_{K}MA_{\omega}(V_{K,\psi_{\min}}^{*})\leq \mathrm{Cap}_{\psi_{\min}}(K) 
$$
where the second equality is the content of Lemma $3.5$ in \cite{DDNL18b}. Therefore, as $\mathrm{Cap}_{\psi}(K)\leq V_{\psi}$ by definition, we have
$$
\mathrm{Cap}_{\psi}(K)\leq \frac{V_{\psi}}{V_{\psi_{\min}}}\mathrm{Cap}_{\psi_{\min}}(K)\leq \frac{V_{\psi}}{V_{\psi_{\min}}^{1/n}}\mathrm{Cap}_{\psi_{\min}}(K)^{1/n}
$$
where clearly the last inequality follows from $\mathrm{Cap}_{\psi_{\min}}(K)\leq V_{\psi_{\min}}$.\newline
Finally, the inequality $\mathrm{Cap}_{\psi_{\min}}(K)\leq C'\mathrm{Cap}(K)^{1/n}$ (see again Lemma $3.2$ in \cite{Lu20}, where the constant $C'$ depends on $X,\omega,\psi_{\min}$) concludes the proof.
\end{proof}
\begin{proof}[Proof of Proposition \ref{lem:Key}]
$MA_{\omega}(u_{1,k}^{j},u_{2,k}^{n-j})\to MA_{\omega}(u_{1}^{j},u_{2}^{n-j})$ weakly as a consequence of Theorem \ref{thm:conv} and Lemma \ref{lem:Volumes1}.\\
To prove $(4)$ for $j=0,\dots,n$ fixed, we define $T_{k}:=MA_{\omega}\big(u_{1,k}^{j},u_{2,k}^{n-j}\big)$ and $T:=MA_{\omega}\big(u_{1}^{j},u_{2}^{n-j}\big)$. If $V_{\psi}=0$ then it is immediate to observe that $f_{k} T_{k}$ converges weakly to $0$, thus we may assume $\psi\in\mathcal{M}^{+}$. By Hartogs' Lemma we can also assume that $\sup_{X}u_{1,k}=\sup_{X}u_{2,k}=0$ for any $k\in\mathbbm{N}$, which clearly implies $u_{1,k},u_{2,k}\leq \psi_{k}$ for any $k\in\mathbbm{N}$ as $u_{j,k}-\sup_{X}u_{j,k}\leq P_{\omega}[\psi_{k}]=\psi_{k}$ (see also Lemma \ref{lem:Bound} below). In particular, by a simple calculation we obtain
\begin{equation}
    \label{eqn:Uffy}
    T_{k}\leq 2^{n}MA_{\omega}\Big(\frac{u_{1,k}+u_{2,k}}{2}\Big)\leq (2C)^{n}MA_{\omega}\Big(\frac{u_{1,k}+u_{2,k}}{2C}+\big(1-\frac{1}{C}\big)\psi_{k}\Big)\leq (2C)^{n}\mathrm{Cap}_{\psi_{k}}
\end{equation}
where the last inequality follows from the hypothesis $u_{1,k},u_{2,k}\geq \psi_{k}-C$.\newline
Then, we fix $\chi$ continuous function and we start with
\begin{equation}
    \label{eqn:Uffy2}
    \Bigg|\int_{X}\chi\big(f_{k}T_{k}-fT\big)\Bigg|\leq \Bigg|\int_{X}\chi\big(f_{k}-f\big)T_{k}\Bigg|+\Bigg|\int_{X}\chi f\big(T_{k}-T\big)\Bigg|.
\end{equation}
For the first term on the right-hand side of (\ref{eqn:Uffy2}) we fix $\delta>0$ and observe that (\ref{eqn:Uffy}) implies
\begin{multline*}
    \Bigg|\int_{X}\chi\big(f_{k}-f\big)T_{k}\Bigg|\leq ||\chi||_{\infty}\int_{\{|f_{k}-f|\geq \delta\}}|f_{k}-f|T_{k}+\delta||\chi||_{\infty}\int_{\{|f_{k}-f|<\delta\}}T_{k}\leq\\\leq (||f_{k}||_{\infty}+||f||_{\infty})||\chi||_{\infty}(2C)^{n}\mathrm{Cap}_{\psi_{k}}\big(\{|f_{k}-f|\geq \delta\}\big)+\delta||\chi||_{\infty}V_{\psi_{k}}\leq\\
    \leq C_{1}\mathrm{Cap}\big(\{|f_{k}-f|\geq \delta\}\big)^{1/n^{2}}+\delta C_{2} 
\end{multline*}
where in the last inequality we used Lemma \ref{lem:AfterReferee} and we denote by $C_{1}, C_{2}$ two uniform constants. Hence, since $f_{k}\to f$ in capacity, we get that the first term in (\ref{eqn:Uffy2}) goes to $0$ as $k\to \infty$.\newline
For the second term in (\ref{eqn:Uffy2}) we fix $\epsilon>0$, $U_{\epsilon}$ open set of $X$ such that $\mathrm{Cap}(U_{\epsilon})<\epsilon$ and $g$ continuous function such that $g=f$ on $X\setminus U_{\epsilon}$. Thus
\begin{multline}
\label{eqn:Uffy3}
    \Bigg|\int_{X}\chi f \big(T_{k}-T\big)\Bigg|\leq \Bigg|\int_{X}\chi g (T_{k}-T)+\int_{U_{\epsilon}}\chi (f-g)(T_{k}-T)\Bigg|\leq \\
    \leq \Bigg|\int_{X}\chi g (T_{k}-T)\Bigg|+||\chi||_{\infty}\big(||f||_{\infty}+||g||_{\infty}\big) \int_{U_{\epsilon}}(T_{k}+T)\leq \\
    \leq \Bigg|\int_{X}\chi g (T_{k}-T)\Bigg|+ C_{3} \mathrm{Cap}(U_{\epsilon})^{1/n^{2}}\leq \Bigg|\int_{X}\chi g (T_{k}-T)\Bigg|+C_{3}\epsilon^{1/n^{2}}
\end{multline}
where we again combined Lemma \ref{lem:AfterReferee} with (\ref{eqn:Uffy}). The proof of (\ref{eqn:Label}) concludes observing that the first term on the right-hand side of (\ref{eqn:Uffy3}) goes to $0$ as $k\to +\infty$ by the weak convergence $T_{k}\to T$.
\end{proof}
\subsection{The $\psi$-relative finite energy class $\mathcal{E}^{1}(X,\omega,\psi)$.}
From now until section \ref{sec:5} we will assume $\psi$ model type envelope and $V_{\psi}>0$, i.e. $\psi\in\mathcal{M}^{+}$ with the notations of the Introduction.
\begin{defn}\emph{(\cite{DDNL17b})}
The $\psi$-relative energy functional $E_{\psi}:PSH(X,\omega,\psi)\to \mathbbm{R}\cup\{-\infty\}$ is defined as
$$
E_{\psi}(u):=\frac{1}{(n+1)}\sum_{j=0}^{n}\int_{X}(u-\psi)MA_{\omega}(u^{j},\psi^{n-j})
$$
if $u$ has $\psi$-relative minimal singularities, and as
$$
E_{\psi}(u):=\inf\{E_{\psi}(v)\, : \, v\in\mathcal{E}(X,\omega,\psi) \, \mathrm{with}\, \psi\mathrm{-} \, \mathrm{relative} \, \mathrm{minimal}\, \mathrm{singularities}, v\geq u\}
$$
otherwise.
\end{defn}
When $\psi=0$ this functional is, up to a multiplicative constant, the \emph{Aubin-Mabuchi energy functional}, also called \emph{Monge-Amp\'ere energy} (see \cite{Aub84}, \cite{Mab86}).
\begin{rem}
\label{rem:Xia}
\emph{The authors in \cite{DDNL17b} introduced this functional and proved some of its properties in their section $\S 4.2$ assuming $\psi$ with small unbounded locus. Indeed, to prove their Theorem $4.10$ (from which all the other results of their section $\S 4.2$ follow) they needed an integration by parts formula and the weak convergence property for the Monge-Ampère measures described by their Lemma $4.1$. However, the integration by parts formula showed by Xia in \cite{X19a} (see also Theorem $1.2$ in \cite{Lu20}) and the weak convergence described in the Proposition \ref{lem:Key} allows to extend Theorem $4.10$ in \cite{DDNL17b} to the general setting of not necessarily small unbounded loci. As an immediate consequence all the results in section $\S 4.2$ in \cite{DDNL17b} extend to the general setting with no changes in the proofs, and some of this results will be recalled in the sequel.}
\end{rem}
By Lemma $4.12$ in \cite{DDNL17b} $E_{\psi}(u)=\lim_{j\to \infty} E_{\psi}(u_{j})$ for arbitrary $u\in PSH(X,\omega,\psi)$ where $u_{j}:=\max(u,\psi-j)$ are the \emph{$\psi$-relative canonical approximants}. Moreover, following the notation in \cite{DDNL17b}, we recall that
$$
\mathcal{E}^{1}(X,\omega,\psi):=\{u\in\mathcal{E}(X,\omega,\psi)\, :\, E_{\psi}(u)>-\infty\}
$$
and that $E_{\psi}(u)>-\infty$ is equivalent to $V_{u}=V_{\psi}$ and $\int_{X}(u-\psi)MA_{\omega}(u)>-\infty$ (compare also Lemma $4.13$ in \cite{DDNL17b} and Proposition $2.11$ in \cite{BEGZ10}).
\begin{prop}\emph{(Section $4.2$ in \cite{DDNL17b})}
\label{prop:Decr}
The $\psi$-relative energy functional is non-decreasing, concave along affine curves and continuous along decreasing sequences. 
\end{prop}
Moreover we also have the following properties:
\begin{prop}
\label{prop:4.10}
Suppose $u,v\in \mathcal{E}^{1}(X,\omega,\psi)$. Then:
\begin{itemize}
\item[i)] $E_{\psi}(u)-E_{\psi}(v)=\frac{1}{(n+1)}\sum_{j=0}^{n}\int_{X}(u-v)MA_{\omega}(u^{j},v^{n-j})$;
\item[ii)] if $u\leq v$ then $\int_{X}(u-v)MA_{\omega}(u)\leq E_{\psi}(u)-E_{\psi}(v)\leq \frac{1}{n+1}\int_{X}(u-v)MA_{\omega}(u)$;
\item[iii)] $\int_{X}(u-v)MA_{\omega}(u)\leq E_{\psi}(u)-E_{\psi}(v)\leq \int_{X}(u-v)MA_{\omega}(v)$.
\end{itemize}
\end{prop}
\begin{proof}
If $u,v$ have $\psi$-relative minimal singularities then it is part of the content of Theorem $4.10$ in \cite{DDNL17b} (which, as specified in Remark \ref{rem:Xia}, can be immediately extended to the general case of not small unbounded loci), while in the general case the proof is the same to that of Proposition $2.2$ in \cite{DDNL18a} replacing $V_{\theta}$ with $\psi$, using the Comparison Principle of Proposition \ref{prop:Comp} and the fact that for any $w\in \mathcal{E}^{1}(X,\omega,\psi)$
\begin{multline*}
\lim_{j\to\infty}j\int_{\{w\leq \psi-j\}}MA_{\omega}(\max(w,\psi-j))=\lim_{j\to \infty}j\int_{\{w\leq \psi-j\}}MA_{\omega}(w)\leq\lim_{j\to\infty}\int_{\{w\leq \psi-j\}}(\psi-w)MA_{\omega}(w)=0
\end{multline*}
since $\int_{X}(w-\psi)MA_{\omega}(w)>-\infty$.
\end{proof}
We conclude the subsection showing that the envelope operator $P_{\omega}(\cdot,\cdot)$ is an operator of the class $\mathcal{E}^{1}(X,\omega,\psi)$ (in the absolute setting, this problem was addressed by Darvas, see Corollary $3.5$ in \cite{Dar15}).
\begin{prop}
\label{prop:Proj}
Assume $u,v\in \mathcal{E}^{1}(X,\omega,\psi)$. Then $P_{\omega}(u,v)\in\mathcal{E}^{1}(X,\omega,\psi)$. Moreover if $\{u_{j},v_{j}\}_{j\in\mathbbm{N}}\subset \mathcal{E}^{1}(X,\omega,\psi)$ decreasing respectively to $u,v\in\mathcal{E}^{1}(X,\omega,\psi)$, then $E_{\psi}\big(P_{\omega}(u_{j},v_{j})\big)$ $\searrow E_{\psi}\big(P_{\omega}(u,v)\big)$.
\end{prop}
\begin{proof}
Up to rescaling we may assume $u,v\leq 0$. For any $j\in\mathbbm{N}$ let $u_{j}:=\max(u,\psi-j), v_{j}:=\max(v,\psi-j)$ be the $\psi$-relative canonical approximants of $u,v$. Then $w_{j}:=P_{\omega}(u_{j},v_{j})$ is a decreasing sequence of potentials with $\psi$-relative minimal singularities. Moreover it is easy to check that $w_{j}\searrow P_{\omega}(u,v)$. Thus by Proposition \ref{prop:Decr} it is sufficient to find an uniform bound for $E_{\psi}(w_{j})$, and by Proposition \ref{prop:4.10} this is equivalent to finding $C>0$ independent of $j$ such that $\int_{X}(\psi-w_{j})MA_{\omega}(w_{j})\leq C$. But Lemma \ref{lem:PMass} implies
\begin{multline*}
\int_{X}(\psi-w_{j})MA_{\omega}(w_{j})\leq \int_{\{w_{j}=u_{j}\}}(\psi-u_{j})MA_{\omega}(u_{j})+\int_{\{w_{j}=v_{j}\}}(\psi-v_{j})MA_{\omega}(v_{j})\leq \\
\leq(n+1)|E_{\psi}(u_{j})+E_{\psi}(v_{j})|\leq(n+1)|E_{\psi}(u)+E_{\psi}(v)|.
\end{multline*}
The second statement is now an easy consequence of the monotonicity of $E_{\psi}$ since $P_{\omega}(u_{j},v_{j})\searrow P_{\omega}(u,v)$ for any couple of decreasing sequences $u_{j}\searrow u,v_{j}\searrow v$.
\end{proof}
\section{A metric geometry on $\mathcal{E}^{1}(X,\omega,\psi)$: proof of Theorem \ref{thmA}.}
\label{sec:ThmA}
Recall that we are assuming $\psi\in \mathcal{M}^{+}$, i.e. $\psi$ model type envelope with $V_{\psi}>0$.
\subsection{$\mathcal{E}^{1}(X,\omega,\psi)$ as metric space.}
In this subsection we prove that $(\mathcal{E}^{1}(X,\omega,\psi),d)$ is a metric space where $d:\mathcal{E}^{1}(X,\omega,\psi)\times\mathcal{E}^{1}(X,\omega,\psi)\to \mathbbm{R}_{\geq 0}$ is defined as
$$
d(u_{1},u_{2}):=E_{\psi}(u_{1})+E_{\psi}(u_{2})-2E_{\psi}(P_{\omega}(u_{1},u_{2})).
$$
It follows from section \ref{sec:Preli} that $d$ assumes finite non-negative values, and that $d$ is continuous along decreasing sequences converging to elements in $\mathcal{E}^{1}(X,\omega,\psi)$. 
\begin{lem}
\label{lem:Properties}
Assume $u,v,w\in\mathcal{E}^{1}(X,\omega,\psi)$. Then the followings hold:
\begin{itemize}
\item[i)] $d(u,v)=d(v,u)$;
\item[ii)] if $u\leq v$ then $d(u,v)=E_{\psi}(v)-E_{\psi}(u)$;
\item[iii)] if $u\leq v\leq w$ then $d(u,w)=d(u,v)+d(v,w)$;
\item[iv)] $d(u,v)=d(u,P_{\omega}(u,v))+d(v,P_{\omega}(u,v))$;
\item[v)] $d(u,v)=0$ iff $u=v$.
\end{itemize}
\end{lem}
\begin{proof}
All points are straightforward except one implication in $(v)$. Thus assume $d(u,v)=0$. Then by $(ii)$ and $(iv)$ we get $E_{\psi}(u)=E_{\psi}(P_{\omega}(u,v))$, which implies $P_{\omega}(u,v)=u$ a.e. with respect to $MA_{\omega}(P_{\omega}(u,v))$ (Proposition \ref{prop:4.10}). Hence by the domination principle (Proposition $3.11$ in \cite{DDNL17b}) we obtain $P_{\omega}(u,v)\geq u$, i.e. $P_{\omega}(u,v)=u$. The conclusion follows by symmetry.  
\end{proof}
To prove that $\mathcal{E}^{1}(X,\omega,\psi)$ is a metric space, it remains to prove the triangle inequality. We proceed as in section $3.1$ in \cite{DDNL18a}, but as a courtesy to the reader we will report here many of their proofs adapted to our setting.
\begin{prop}
\label{prop:1b}
Let $u,v\in \mathcal{E}^{1}(X,\omega,\psi)$ be potentials with $\psi$-relative minimal singularities. For $t\in[0,1]$ set $\varphi_{t}:=P_{\omega}\big((1-t)u+tv,v\big)$. Then for any $t\in[0,1]$
$$
\frac{d}{dt}E_{\psi}(\varphi_{t})=\int_{X}\big(v-\min(u,v)\big)MA_{\omega}(\varphi_{t}).
$$
\end{prop}
\begin{proof}
Let us prove the formula for the right derivative. The same argument easily works for the left derivative. Thus fix $t\in[0,1)$, let $s>0$ small and $f_{t}:=\min\big((1-t)u+tv,v\big)$. Using Proposition \ref{prop:4.10}$.(iii)$ and Lemma \ref{lem:PMass} it is easy to check that
$$
\int_{X}(f_{t+s}-f_{t})MA_{\omega}(\varphi_{t+s})\leq E_{\psi}(\varphi_{t+s})-E_{\psi}(\varphi_{t})\leq \int_{X}(f_{t+s}-f_{t})MA_{\omega}(\varphi_{t}).
$$
Moreover $\varphi_{t+s}\to \varphi_{t}$ uniformly as $s\to 0^{+}$ since $||u-v||_{L^{\infty}}\leq C$, thus $MA_{\omega}(\varphi_{t+s})$ converges weakly to $MA_{\omega}(\varphi_{t})$. Therefore since $f_{t+s}-f_{t}=s\big(v-\min(u,v)\big)$ and since and again $||u-v||_{L^{\infty}}\leq C$, Theorem \ref{thm:conv} yields
$$
\lim_{s\to 0^{+}}\frac{E_{\psi}(\varphi_{t+s})-E_{\psi}(\varphi_{t})}{s}=\int_{X}\big(v-\min(u,v)\big)MA_{\omega}(\varphi_{t}).
$$
\end{proof}
\begin{prop}
\label{prop:3b}
Let $u,v\in\mathcal{E}^{1}(X,\omega,\psi)$. Then $d\big(\max(u,v),u\big)\geq d\big(v,P_{\omega}(u,v)\big)$.
\end{prop}
\begin{proof}
Setting $\varphi:=\max(u,v)$ and $\phi:=P_{\omega}(u,v)$, the inequality to prove is equivalent to $E_{\psi}(v)-E_{\psi}(\phi)\leq E_{\psi}(\varphi)-E_{\psi}(u)$. By Proposition \ref{prop:Decr} we may assume $u,v$ having $\psi$-relative minimal singularities.\\
Next Proposition \ref{prop:1b} implies
\begin{multline*}
E_{\psi}(\varphi)-E_{\psi}(u)=\int_{0}^{1}\int_{X}(\varphi-u)MA_{\omega}\big((1-t)u+t\varphi\big)dt=\int_{0}^{1}\int_{X}(\varphi-u)MA_{\omega}\big(\max(w_{t},u)\big)=\\
=\int_{0}^{1}\int_{\{v>u\}}(v-u)MA_{\omega}(w_{t})
\end{multline*}
for $w_{t}:=(1-t)u+tv$ for $t\in[0,1]$, and where the last equality follows from the locality of the Monge-Ampère operator with respect to the plurifine topology.\\
On the other hand combining Proposition \ref{prop:1b} with Lemma \ref{lem:PMass} and $\{w_{t}\leq v\}=\{u\leq v\}$ we get
$$
E_{\psi}(v)-E_{\psi}(\phi)=\int_{0}^{1}\int_{X}\big(v-\min(u,v)\big)MA_{\omega}\big(P_{\omega}(w_{t},v)\big)\leq \int_{0}^{1}\int_{\{v>u\}}(v-u)MA_{\omega}(w_{t}),
$$
which concludes the proof.
\end{proof}
\begin{cor}
\label{cor:4b}
Let $u,v,w\in\mathcal{E}^{1}(X,\omega,\psi)$. Then $d(u,v)\geq d\big(P_{\omega}(u,w),P_{\omega}(v,w)\big)$.
\end{cor}
\begin{proof}
It follows from Lemma \ref{lem:Properties}$.(iii)$ and Proposition \ref{prop:3b} by an easy calculation (see Corollary $3.5$ in \cite{DDNL18a} for the details).
\end{proof}
We are now ready to prove the main theorem of this subsection:
\begin{thm}
$(\mathcal{E}^{1}(X,\omega,\psi),d)$ is a metric space.
\end{thm}
\begin{proof}
As said before, it remains only to prove the triangle inequality (see Lemma \ref{lem:Properties}).\\
Let $u,v,w\in\mathcal{E}^{1}(X,\omega,\psi)$ and observe that the inequality $d(u,v)\leq d(u,w)+d(w,v)$ is equivalent to 
$$
E_{\psi}\big(P_{\omega}(u,w)\big)-E_{\psi}\big(P_{\omega}(u,v)\big)\leq E_{\psi}(w)-E_{\psi}\big(P_{\omega}(w,v)\big).
$$
By Corollary \ref{cor:4b} and using the monotonicity of the $\psi$-relative energy functional (Proposition \ref{prop:Decr}) we get
\begin{multline*}
E_{\psi}(w)-E_{\psi}\big(P_{\omega}(w,v)\big)=d\big(w,P_{\omega}(w,v)\big)\geq d\big(P_{\omega}(w,u), P_{\omega}(w,v,u)\big)=\\
=E_{\psi}\big(P_{\omega}(w,u)\big)-E_{\psi}\big(P_{\omega}(w,v,u)\big)\geq E_{\psi}\big(P_{\omega}(w,u)\big)-E_{\psi}\big(P_{\omega}(u,v)\big),
\end{multline*}
which implies the Theorem.
\end{proof}
\subsection{Completeness of $(\mathcal{E}^{1}(X,\omega,\psi),d)$.}
To show the completeness we first need to extend some results known in the absolute setting (i.e. if $\psi=0$, see \cite{BEGZ10}, \cite{DDNL18a}).
\begin{prop}
\label{prop:Ineq}
Assume $u,v\in\mathcal{E}^{1}(X,\omega,\psi)$. Then
$$
\frac{1}{3\cdot 2^{n+2}(n+1)}\Big(\int_{X}|u-v|\big(MA_{\omega}(u)+MA_{\omega}(v)\big)\Big)\leq d(u,v)\leq \int_{X}|u-v|\big(MA_{\omega}(u)+MA_{\omega}(v)\big).
$$
\end{prop}
\begin{proof}
The proof is the same as that of Theorem $3.7$ in \cite{DDNL18a} replacing their Theorem $2.1$ and Lemma $3.1$ by our Proposition \ref{prop:4.10} and Lemma \ref{lem:Properties}.
\end{proof}
\begin{lem}
\label{lem:Bound}
There exist positive constants $A>1,B>0$ such that for any $u\in\mathcal{E}^{1}(X,\omega,\psi)$
$$
-d(\psi,u)\leq V_{\psi}\sup_{X}(u-\psi)=V_{\psi}\sup_{X}u\leq Ad(\psi,u)+B.
$$
\end{lem}
\begin{proof}
The equality follows from $u-\sup_{X}u\leq P_{\omega}[\psi]=\psi\leq 0$.\\
Next, if $\sup_{X}u\leq 0$ then the right inequality is trivial for any $A,B>0$ while the left inequality is a consequence of $d(\psi,u)=-E_{\psi}(u)\geq -V_{\psi}\sup_{X}(u-\psi)$ (Proposition \ref{prop:4.10}).\\
Therefore, let us assume $\sup_{X}u\geq 0$. By Proposition $2.7$ in \cite{GZ05} there exists an uniform bound $C>0$ such that
$$
\int_{X}|v-\sup_{X}v|MA_{\omega}(0)\leq C
$$
for any $v\in PSH(X,\omega)$. Hence, since by Theorem \ref{thm:3.8} $MA_{\omega}(\psi)\leq \mathbbm{1}_{\{\psi=0\}} MA_{\omega}(0)$, we also have
$$
\int_{X}|u-\sup_{X}u-\psi|MA_{\omega}(\psi)\leq\int_{X}|u-\sup_{X}u|MA_{\omega}(0)\leq C
$$
for any $u\in\mathcal{E}^{1}(X,\omega,\psi)$. So, by Proposition \ref{prop:Ineq},
\begin{multline*}
d(u,\psi)\geq D\int_{X}|u-\psi|MA_{\omega}(\psi)\geq DV_{\psi}\sup_{X}u-D\int_{X}|u-\sup_{X}u-\psi|MA_{\omega}(\psi)\geq DV_{\psi}\sup_{X}u-DC.
\end{multline*}
Take $A:=1/D$ and $B:=C$ to conclude the proof.
\end{proof}
\begin{prop}
\label{prop:Incr}
Let $\{u_{j}\}_{j\in\mathbbm{N}}\subset \mathcal{E}^{1}(X,\omega,\psi)$ be an increasing sequence uniformly bounded by above, and let $u:=\big(\lim_{j\to \infty}u_{j}\big)^{*}\in PSH(X,\omega)$. Then $u\in \mathcal{E}^{1}(X,\omega,\psi)$ and $E_{\psi}(u_{j})\to E_{\psi}(u)$ as $j\to \infty$.
\end{prop}
\begin{proof}
Since $\sup_{X}u_{j}=\sup_{X}(u_{j}-\psi)$ is uniformly bounded, we immediately get that $u\leq \psi+C$ for a certain constant $C\in\mathbbm{R}$, i.e. $u\in PSH(X,\omega,\psi)$. Furthermore since $u\geq u_{j}$ for any $j\in\mathbbm{N}$ by construction, we also obtain $u\in\mathcal{E}^{1}(X,\omega,\psi)$ by \cite{WN17} and the monotonicity of $E_{\psi}$. Thus since $E_{\psi}(u)=\lim_{k\to \infty}E_{\psi}(u^{k})$, $E_{\psi}(u_{j})=\lim_{k\to \infty}E_{\psi}(u_{j}^{k})$ where $u^{k}:=\max(u,\psi-k)$, $u_{j}^{k}:=\max(u_{j},\psi-k)$ are the $\psi$-relative canonical approximants, it is enough to check that $E_{\psi}(u_{j}^{k})\searrow E_{\psi}(u_{j})$ as $k\to \infty$ uniformly in
$j$. Indeed this would imply
$$
E_{\psi}(u)-E_{\psi}(u_{j})\leq |E_{\psi}(u)-E_{\psi}(u^{k})|+|E_{\psi}(u^{k})-E_{\psi}(u_{j}^{k})|+|E_{\psi}(u_{j}^{k})-E_{\psi}(u_{j})|\to 0
$$
letting first $j\to \infty$ and then $k\to \infty$, since $|E_{\psi}(u^{k})-E_{\psi}(u_{j}^{k})|\to 0$ as $j\to \infty$ as a consequence of Lemma $4.1$ in \cite{DDNL17b} (see also Proposition \ref{lem:Key}).\\
Assume without loss of generality that $u\leq 0$. By Proposition \ref{prop:4.10} we have
\begin{equation}
\label{eqn:Last}
0\leq E_{\psi}(u_{j}^{k})-E_{\psi}(u_{j})\leq \int_{X}(u_{j}^{k}-u_{j})MA_{\omega}(u_{j})=\int_{k}^{+\infty}MA_{\omega}(u_{j})\big(\{u_{j}\leq \psi-t\}\big) dt.
\end{equation}
Next we set $v_{j,t}:=\frac{u_{j}+\psi-t}{2}$ and we note that the following inclusions hold:
$$
\{u_{j}\leq \psi-t\}\subset \{u_{1}\leq v_{j,t}\}\subset \{u_{1}\leq \psi-t/2\}.
$$
Indeed the first inclusion follows from $u_{1}\leq u_{j}$ while the last is a consequence $\sup_{X}u=\sup_{X}(u-\psi)$ (Lemma \ref{lem:Bound}). Thus, by the comparison principle (Proposition \ref{prop:Comp}) we have
\begin{multline*}
MA_{\omega}(u_{j})\big(\{u_{j}\leq \psi-t\}\big)\leq MA_{\omega}(u_{j})\big(\{u_{1}\leq v_{j,t}\}\big)\leq 2^{n}MA_{\omega}(v_{j,t})\big(\{u_{1}\leq v_{j,t}\}\big)\leq \\
\leq2^{n}MA_{\omega}(u_{1})\big(\{u_{1}\leq v_{j,t}\}\big)\leq 2^{n}MA_{\omega}(u_{1})\big(\{u_{1}\leq \psi-t/2\}\big).
\end{multline*}
Therefore, continuing the estimates in (\ref{eqn:Last}),
$$
0\leq E_{\psi}(u_{j}^{k})-E_{\psi}(u_{j})\leq 2^{n+1}\int_{k/2}^{+\infty}MA_{\omega}(u_{1})\big(\{u_{1}\leq \psi-t\}\big) dt=2^{n+1}\int_{X}(u_{1}^{k/2}-u_{1})MA_{\omega}(u_{1}),
$$
which concludes the proof since the right hand goes to $0$ as $k\to +\infty$ (recall that $u_{1}\in\mathcal{E}^{1}(X,\omega,\psi)$).
\end{proof}
\begin{thm}
$(\mathcal{E}^{1}(X,\omega,\psi),d)$ is complete.
\end{thm}
\begin{proof}
Let $\{u_{j}\}_{j\in\mathbbm{N}}\subset (\mathcal{E}^{1}(X,\omega,\psi),d)$ be a Cauchy sequence. Up to extract a subsequence we may assume that $d(u_{j},u_{j+1})\leq 2^{-j}$ for any $j\in \mathbbm{N}$. Define $v_{j,k}:=P_{\omega}(u_{j},\dots,u_{k})$ for $j,k\in\mathbbm{N}, k\geq j$, i.e.
$$
v_{j,k}=\sup\{v\in PSH(X,\omega)\,:\, v\leq \min (u_{j},\dots,u_{k})\}^{*}.
$$
Clearly $v_{j,k}=P_{\omega}(u_{j},v_{j+1,k})\leq v_{j+1,k}$ if $k\geq j+1$ and $v_{j,k}\in \mathcal{E}^{1}(X,\omega,\psi)$ as consequence of Proposition \ref{prop:Proj} since $P_{\omega}(P_{\omega}(u,v),w)=P_{\omega}(u,v,w)$ for any $u,v,w\in PSH(X,\omega)$. Thus for any $k\geq j+1$
$$
d(u_{j},v_{j,k})=d(u_{j},P_{\omega}(u_{j},v_{j+1,k}))\leq d(u_{j},v_{j+1,k})\leq d(u_{j},u_{j+1})+d(u_{j+1},v_{j+1,k})
$$
using Lemma \ref{lem:Properties}. Iterating the argument we get
$$
d(u_{j},v_{j,k})\leq \sum_{s=1}^{k-j}d(u_{j+s-1},u_{j+s})\leq \sum_{s=1}^{k-j}\frac{1}{2^{j+s-1}}\leq \sum_{s=j}^{\infty}\frac{1}{2^{s}}=\frac{1}{2^{j-1}}.
$$
Moreover, since $v_{j,k}$ is decreasing in $k$, there exists a constant $C_{j}\in\mathbbm{R}$ such that $v_{j,k}\leq \psi+C_j{}$ for any $k\geq j$. So
$$
C_{j}-E_{\psi}(v_{j,k})=d(\psi+C_{j},v_{j,k})\leq d(\psi+C_{j},u_{j})+d(u_{j},v_{j,k})\leq d(\psi+C_{j},u_{j})+2^{-j+1},
$$
which implies that $v_{j}:=\lim_{k\to\infty}v_{j,k}\in \mathcal{E}^{1}(X,\omega,\psi)$ by Proposition \ref{prop:Decr}, and $d(v_{j},u_{j})\leq 2^{-j+1}$ by continuity of the distance along decreasing sequences.\\
Next we observe that $v_{j}$ is increasing in $j$ and that
\begin{multline*}
V_{\psi}\sup_{X}(v_{j}-\psi)=V_{\psi}\sup_{X}v_{j}\leq Ad(\psi,v_{j})+B\leq\\
\leq A\Big(d(\psi,u_{1})+\sum_{s=1}^{j-1}d(u_{s},u_{s+1})+d(u_{j},v_{j})\Big)+B\leq Ad(\psi,u_{1})+3A+B
\end{multline*}
where the first inequality is the content of Lemma \ref{lem:Bound}. Hence Proposition \ref{prop:Incr} leads to $u:=\big(\lim_{j\to\infty}v_{j}\big)^{*}\in \mathcal{E}^{1}(X,\omega,\psi)$ and to
$$
d(u_{j},u)\leq d(u_{j},v_{j})+d(v_{j},u)\leq 2^{-j+1}+d(v_{j},u)\to 0
$$
for $j\to \infty$.
\end{proof}
\begin{rem}
\label{rem:ZeroMass}
\emph{In the case $\psi\in\mathcal{M}\setminus\mathcal{M}^{+}$, i.e. if $\psi$ is a model type envelope with zero mass $V_{\psi}=0$, then
$$
PSH(X,\omega,\psi)=\mathcal{E}(X,\omega,\psi)=\mathcal{E}^{1}(X,\omega,\psi).
$$
Indeed any function $u\in PSH(X,\omega,\psi)$ has zero mass $V_{u}=0$ (\cite{WN17}). Thus $E_{\psi}(u)=0$ by definition as for any $\phi$ with $\psi$-minimal singularities we clearly have $|E_{\psi}(\phi)|\leq V_{\psi}\sup_{X}|\phi-\psi|=0$. In particular $d(u,v):=E_{\psi}(u)+E_{\psi}(v)-2E_{\psi}\big(P_{\omega}(u,v)\big)=0$ for any $u,v\in\mathcal{E}^{1}(X,\omega,\psi)$.\\
Moreover if $P_{\omega}\big(\psi_{1},\psi_{2}\big)\not\equiv -\infty$ for $\psi_{1},\psi_{2}$ model type envelopes with zero masses $V_{\psi_{1}}=V_{\psi_{2}}=0$ then $\mathcal{E}^{1}(X,\omega,\psi_{1})\cap\mathcal{E}^{1}(X,\omega,\psi_{2})=\mathcal{E}^{1}(X,\omega,P_{\omega}[P_{\omega}(\psi_{1},\psi_{2})])$ is not empty.}
\end{rem}
\section{The metric space $(X_{\mathcal{A}},d_{\mathcal{A}})$ and consequences.}
\label{sec:5}
In this section we will prove the main Theorem \ref{thmC}, i.e. assuming $\mathcal{A}\subset \mathcal{M}^{+}$ totally ordered subset (recall the the partial order $\preccurlyeq$ coincides with the order $\leq$ on $\mathcal{M}$), we will endow the space $X_{\mathcal{A}}:=\bigsqcup_{\psi\in\mathcal{\overline{A}}}\mathcal{E}^{1}(X,\omega,\psi)$ with a metric topology. Here $\mathcal{\overline{A}}$ denotes the closure of $\mathcal{A}$ as subset of $PSH(X,\omega)$ with its $L^{1}$-topology. \\
We will show that $\mathcal{\overline{A}}\subset \mathcal{M}$ and we will define a natural distance $d_{\mathcal{A}}$ on $X_{\mathcal{A}}$ which extends the distance $d$ (Theorem \ref{thmA}) on $\mathcal{E}^{1}(X,\omega,\psi)$ for any $\psi\in\mathcal{\overline{A}}$ where if $\psi=\inf \mathcal{A}\in\mathcal{M}\setminus \mathcal{M}^{+}$ we identify the space $\mathcal{E}^{1}(X,\omega,\psi)$ with a point as $d_{\mathcal{A}}=d\equiv 0$ (Remark \ref{rem:ZeroMass}).\\
We recall that the distance $d$ on $\mathcal{E}^{1}(X,\omega,\psi)$ for $\psi\in\mathcal{M}^{+}$ is defined as
$$
d(u,v)=E_{\psi}(u)+E_{\psi}(v)-2E_{\psi}\big(P_{\omega}(u,v)\big).
$$
\begin{defn}
\label{defn:Strong}
Given $\psi\in \mathcal{M}^{+}$, the \emph{strong topology} on $\mathcal{E}^{1}(X,\omega,\psi)$ is defined as the metric topology given by the distance $d$.
\end{defn}
In the case $\psi=0$ the strong topology was introduced in \cite{BBEGZ16} (Definition $2.1.$), see also Proposition $5.9$ and Theorem $5.5.$ in \cite{Dar14}.
\subsection{The contraction property of $d$.}
\label{ssec:NewDistance}
\begin{lem}
\label{lem:Pro}
Let $\psi,\psi_{1},\psi_{2}\in \mathcal{M}$ such that $\psi_{2}\preccurlyeq\psi_{1}\preccurlyeq\psi$. Then:
\begin{itemize} 
\item[i)] $P_{\omega}[\psi_{2}](P_{\omega}[\psi_{1}](u))=P_{\omega}[\psi_{2}](u)$ for any $u\in \mathcal{E}^{1}(X,\omega,\psi)$;
\item[ii)] $P_{\omega}[\psi_{1}](\mathcal{E}^{1}(X,\omega,\psi))\subset \mathcal{E}^{1}(X,\omega,\psi_{1})$;
\item[iii)] for any $u,v\in\mathcal{E}^{1}(X,\omega,\psi)$ such that $u-v$ is globally bounded, $||P_{\omega}[\psi_{1}](u)-P_{\omega}[\psi_{2}](v)||_{L^{\infty}}\leq ||u-v||_{L^{\infty}}$ and in particular $P_{\omega}[\psi_{1}](u)$ has $\psi_{1}$-relative minimal singularities for any $u\in \mathcal{E}^{1}(X,\omega,\psi)$ with $\psi$-relative minimal singularities.
\end{itemize}
\end{lem}
\begin{proof}
The inequality $P_{\omega}[\psi_{1}](u)\leq u$ immediately implies $P_{\omega}[\psi_{2}](P_{\omega}[\psi_{1}](u))\leq P_{\omega}[\psi_{2}](u)$. Vice versa $P_{\omega}[\psi_{2}](u)\leq P_{\omega}[\psi_{1}](u)$ since $\psi_{2}\preccurlyeq \psi_{1}$. Thus the first point follows applying $P_{\omega}[\psi_{2}](\cdot) $ to both sides.\\
For the third statement, letting $C:=||u-v||_{L^{\infty}}$, it is an easy consequence of the definition that
$$
P_{\omega}[\psi_{1}](u)\geq P_{\omega}[\psi_{1}](v-C)=P_{\omega}[\psi_{1}](v)-C.
$$
By symmetry, we also have $P_{\omega}[\psi_{1}](u)\leq P_{\omega}[\psi_{1}](v)+C$ which gives $(iii)$. This immediately yields $(ii)$ if $\psi_{1}\in\mathcal{M}\setminus \mathcal{M}^{+}$. Therefore it remains to prove $(ii)$ assuming $\psi_{1}\in\mathcal{M}^{+}$. Letting $u_{j}:=\max(u,\psi-j)$ be the $\psi$-relative canonical approximants of a generic $u\in \mathcal{E}^{1}(X,\omega,\psi)$, we get that $v:=P_{\omega}[\psi_{1}](u)$ belongs to $\mathcal{E}^{1}(X,\omega,\psi_{1})$ if and only if $\int_{X}(\psi_{1}-v_{j})MA_{\omega}(v_{j})$ is uniformly bounded in $j$ where $v_{j}:=P_{\omega}[\psi_{1}](u_{j})$ (see Proposition \ref{prop:Decr}). But taking $D:=\sup_{X} u>0$ and using Theorem \ref{thm:3.8} we get
\begin{multline*}
DV_{\psi_{1}}+\int_{X}(\psi_{1}-v_{j})MA_{\omega}(v_{j})\leq \int_{\{v_{j}=u_{j}\}}(\psi+D-u_{j})MA_{\omega}(u_{j})\leq\\
\leq DV_{\psi} +\int_{X}(\psi-u_{j})MA_{\omega}(u_{j})<E\in\mathbbm{R}
\end{multline*}
for an uniform $E\in\mathbbm{R}$ since $u\in \mathcal{E}^{1}(X,\omega,\psi)$ and $u_{j}\searrow u$.
\end{proof}
We are now ready to prove the following key property of the distance $d$.
\begin{prop}
\label{prop:Dist}
Let $\psi,\psi'\in \mathcal{M}$ such that $\psi'\preccurlyeq \psi$. Then the map
$$
P_{\omega}[\psi'](\cdot):\big(\mathcal{E}^{1}(X,\omega,\psi),d\big)\to \big(\mathcal{E}^{1}(X,\omega,\psi'),d\big)
$$
is a Lipschitz map of Lipschitz constant equal to $1$, i.e.
$$
d(u,v)\geq d(P_{\omega}[\psi'](u),P_{\omega}[\psi'](v))
$$
for any $u,v\in \mathcal{E}^{1}(X,\omega,\psi)$.
\end{prop}
\begin{proof}
Let $u,v\in \mathcal{E}^{1}(X,\omega,\psi)$. Set
$$
\rho(u,v):=\int_{X}(u-v)MA_{\omega}(v).
$$
if $u\geq v$ and $\rho(u,v):=\rho(v,u)$ if $v\geq u$. Proposition \ref{prop:4.10} implies $ d(u,v)\leq \rho(u,v)$. Moreover assuming $\psi'\in \mathcal{M}^{+}$ such that $\psi'\preccurlyeq \psi$ as in the statement of the Proposition,
\begin{multline*}
\rho(P_{\omega}[\psi'](u),P_{\omega}[\psi'](v))=\int_{X}\big(P_{\omega}[\psi'](u)-P_{\omega}[\psi'](v)\big)MA_{\omega}\big(P_{\omega}[\psi'](v)\big)\leq\\
\leq \int_{\{P_{\omega}[\psi'](v)=v\}}(u-v)MA_{\omega}(v)\leq \rho(u,v)
\end{multline*}
by Theorem \ref{thm:3.8}. Therefore we define for any $u,v\in\mathcal{E}^{1}(X,\omega,\psi)$
$$
\tilde{d}(u,v):=\inf\Big( \rho(u,w_{1})+ \sum_{j=1}^{m-1}\rho(w_{j},w_{j+1})+\rho(w_{m},v)\Big)
$$
where the infimum is over all chain $\mathcal{C}=\{u=w_{0},w_{1},\dots,w_{m},w_{m+1}=v\}$ for any $m\in \mathbbm{N}$ such that any pair of consecutive elements in the chain $w_{j},w_{j+1}$ satisfies $w_{j}\geq w_{j+1}$ or $w_{j}\leq w_{j+1}$.\\
Clearly $\tilde{d}(u,v)\geq d(u,v)$ and it is straightforward to check that $\tilde{d}$ is symmetric and that it satisfies the triangle inequality. Moreover by construction and Lemma \ref{lem:Pro} we also have $\tilde{d}\big(P_{\omega}[\psi'](u),P_{\omega}[\psi'](v)\big)\leq \tilde{d}(u,v)$ since $\rho$ has the same property and $P_{\omega}[\psi'](u)\leq P_{\omega}[\psi'](v)$ if $u\leq v$. Thus to conclude the proof it remains to prove that $\tilde{d}\leq d$, which would imply $\tilde{d}=d$.\\
We first observe that it is enough to show that $\tilde{d}(u,v)\leq d(u,v)$ assuming $u\geq v$ since it would lead to
\begin{multline*}
\tilde{d}(w_{1},w_{2})\leq \tilde{d}\big(w_{1}, P_{\omega}(w_{1},w_{2})\big)+\tilde{d}\big(w_{2},P_{\omega}(w_{1},w_{2})\big)\leq\\
\leq d\big(w_{1},P_{\omega}(w_{1},w_{2})\big)+d\big(w_{2},P_{\omega}(w_{1},w_{2})\big)=d(w_{1},w_{2})
\end{multline*}
by Lemma \ref{lem:Properties}$.(iv)$. Therefore let $u\geq v$, fix $N\in\mathbbm{N}$ and set $w_{j,N}:=\frac{j}{N}u+\frac{N-j}{N}v$ for $j=0,\dots,N$. Then $\mathcal{C}_{N}:=\{w_{0,N},\dots,w_{N,N}\}$ is an admissible chain for the definition of $\tilde{d}(u,v)$. So
\begin{gather*}
\tilde{d}(u,v)\leq\sum_{j=0}^{N-1}\rho(w_{j,N},w_{j+1,N})=\sum_{j=0}^{N-1}\frac{1}{N}\int_{X}(u-v)MA_{\omega}(w_{j,N})=\\
=\sum_{s=0}^{n}\binom{n}{s}\Bigg(\frac{1}{N}\sum_{j=0}^{N-1}\Big(\frac{j}{N}\Big)^{s}\Big(\frac{N-j}{N}\Big)^{n-s}\Bigg)\int_{X}(u-v)MA_{\omega}(u^{s},v^{n-s}).
\end{gather*}
Next by Lemma \ref{lem:Contazzo} below for any $s=0,\dots,n$,
$$
\frac{1}{N}\sum_{j=0}^{N-1}\Big(\frac{j}{N}\Big)^{s}\Big(\frac{N-j}{N}\Big)^{n-s}\longrightarrow \frac{1}{\binom{n}{s}(n+1)},
$$
as $N\to \infty$. Hence we get
$$
\tilde{d}(u,v)\leq\frac{1}{n+1}\sum_{s=0}^{n}\int_{X}(u-v)MA_{\omega}(u^{s},v^{n-s})=d(u,v),
$$
which concludes the proof.
\end{proof}
\begin{lem}
\label{lem:Contazzo}
Let $n,N\in\mathbbm{N}$ and let $s$ be a non negative integer such that $s\leq n$. Then
\begin{equation}
\label{eqn:Stran}
\lim_{N\to \infty}\frac{1}{N}\sum_{j=0}^{N-1}\Big(\frac{j}{N}\Big)^{s}\Big(\frac{N-j}{N}\Big)^{n-s}=\frac{1}{\binom{n}{s}(n+1)}
\end{equation}
\end{lem}
\begin{proof}
Consider the function $f:[0,1]\to \mathbbm{R}, x\to x^{n-s}(1-x)^{s}$, it is immediate to see that the sequence in (\ref{eqn:Stran}) is the upper Darboux sum of $f$ with respect to the partition $0<\frac{1}{N}<\dots<\frac{j}{N}<\dots< 1$. Thus
$$
\lim_{N\to \infty}\frac{1}{N}\sum_{j=0}^{N-1}\Big(\frac{j}{N}\Big)^{s}\Big(\frac{N-j}{N}\Big)^{n-s}=\int_{0}^{1}x^{n-s}(1-x)^{s}dx.
$$
A brief calculation shows that $\int_{0}^{1}x^{n-s}(1-x)^{s}dx=\frac{1}{\binom{n}{s}(n+1)}$.
\end{proof}
The contraction property showed above implies an uniform convergence on some compact sets.
\begin{prop}
\label{prop:UC}
Let $\{\psi_{k}\}_{k\in\mathbbm{N}}\subset \mathcal{M}^{+}$ be a sequence monotonically converging to $\psi\in \mathcal{M}$ almost everywhere. Then for any $\psi'\in\mathcal{M}$ such that $\psi'\succcurlyeq \psi_{k}$ for any $k\geq k_{0}$ big enough and for any compact set $\tilde{K}\subset \mathcal{E}^{1}(X,\omega,\psi')$ with respect the strong topology on $\mathcal{E}^{1}(X,\omega,\psi')$, the sets $K:=P_{\omega}[\psi](\tilde{K})\subset \big(\mathcal{E}^{1}(X,\omega,\psi),d\big)$, $K_{k}:=P_{\omega}[\psi_{k}](\tilde{K})\subset \big(\mathcal{E}^{1}(X,\omega,\psi_{k}),d\big)$ are compact in their respective strong topologies for any $k\geq k_{0}$, and
$$
d\big(P_{\omega}[\psi_{k}](\varphi_{1}),P_{\omega}[\psi_{k}](\varphi_{2})\big)\to d\big(P_{\omega}[\psi](\varphi_{1}),P_{\omega}[\psi](\varphi_{2})\big)
$$
uniformly on $\tilde{K}\times\tilde{K}$, i.e. varying $(\varphi_{1},\varphi_{2})\in \tilde{K}\times \tilde{K}$.
\end{prop}
\begin{proof}
It follows from Lemma \ref{lem:Pro} and Proposition \ref{prop:Dist} that $P_{\omega}[\psi_{k}](\tilde{K})$ is compact in $\big(\mathcal{E}^{1}(X,\omega,\psi_{k}),d\big)$ for any $k\in\mathbbm{N}$, and similarly for $K$.\\
Next, we define $f_{k}:\tilde{K}\times \tilde{K}\to \mathbbm{R}_{\geq 0}$ for $k\in\mathbbm{N}$ and $f:\tilde{K}\times \tilde{K}\to \mathbbm{R}_{\geq 0}$ respectively as
\begin{gather*}
f_{k}(\varphi_{1},\varphi_{2}):=d\big(P_{\omega}[\psi_{k}](\varphi_{1}),P_{\omega}[\psi_{k}](\varphi_{2})\big)\\
f(\varphi_{1},\varphi_{2}):=d\big(P_{\omega}[\psi](\varphi_{1}),P_{\omega}[\psi](\varphi_{2})\big).
\end{gather*}
We observe that $f_{k},f$ are Lipschitz continuous with respect the strong topology on $\tilde{K}\times \tilde{K}$ (Proposition \ref{prop:Dist}). Moreover by Lemma \ref{lem:NewLemma} below $f_{k}\to f$ pointwise on a dense subset of $\tilde{K}\times \tilde{K}$ and $\{f_{k}\}_{k\in\mathbbm{N}}$ is a monotone sequence. Hence Dini's Theorem implies that $f_{k}\to f$ uniformly on $\tilde{K}\times \tilde{K}$.
\end{proof}
\begin{lem}
\label{lem:NewLemma}
Let $\{\psi_{k}\}_{k\in\mathbbm{N}}\subset \mathcal{M}^{+}$ totally ordered such that $\psi_{k}\to \psi\in\mathcal{M}$ monotonically, and let $\{u_{1,k},u_{2,k}\}_{k\in\mathbbm{N}}$ be two sequences such that $u_{1,k},u_{2,k}\in\mathcal{E}^{1}(X,\omega,\psi_{k})$, and such that $ u_{1,k},u_{2,k}\geq \psi_{k}-C$ uniformly in $k\in\mathbbm{N}$. If either $u_{1,k}\searrow u_{1},u_{2,k}\searrow u_{2}$ a.e. or $u_{1,k}\nearrow u_{1},u_{2,k}\nearrow u_{2}$ a.e., then
$$
d(u_{1,k},u_{2,k})\to d(u_{1},u_{2}).
$$
\end{lem}
\begin{proof}
As immediate consequence of Proposition \ref{lem:Key} and of the definition of $d$, to prove that $d(u_{1,k},u_{2,k})\to d(u_{1},u_{2})$ we only need to show that $P_{\omega}(u_{1,k},u_{2,k})\to P_{\omega}(u_{1},u_{2})$ almost everywhere since clearly $u_{i,k}-P_{\omega}(u_{1,k},u_{2,k})$ is uniformly bounded and $P_{\omega}(u_{1,k},u_{2,k})\geq \psi_{k}-C$. From $P_{\omega}(u_{1,k},u_{1,k})\leq u_{1,k},u_{2,k}$ we immediately have $\big(\lim_{k\to \infty}P_{\omega}(u_{1,k},u_{2,k})\big)^{*}\leq P_{\omega}(u_{1},u_{2})$. Therefore if the convergence is decreasing then $P_{\omega}(u_{1},u_{2})\leq P_{\omega}(u_{1,k},u_{2,k})$ and we get the convergence of the $d$ distances. If instead the convergence is increasing then, setting $\phi:=\big(\lim_{k\to \infty} P_{\omega}(u_{1,k},u_{2,k})\big)^{*}$, we observe that $\phi\in \mathcal{E}(X,\omega,\psi)$ and that $MA_{\omega}\big(P_{\omega}(u_{1,k},u_{2,k})\big)\to MA_{\omega}(\phi)$ weakly as a consequence of Theorem \ref{thm:conv} and Lemma \ref{lem:Volumes1}. Moreover since by Lemma \ref{lem:PMass}
\begin{multline*}
MA_{\omega}\big(P_{\omega}(u_{1,k},u_{2,k})\big)\leq \mathbbm{1}_{\{\phi\geq u_{1,k}\}}MA_{\omega}(u_{1,k})+\mathbbm{1}_{\{\phi\geq u_{2,k}\}}MA_{\omega}(u_{2,k})\leq\\
\leq \mathbbm{1}_{\{\phi\geq u_{1,j}\}}MA_{\omega}(u_{1,k})+\mathbbm{1}_{\{\phi\geq u_{2,j}\}}MA_{\omega}(u_{2,k})
\end{multline*}
for any $j\leq k$, and $MA_{\omega}(u_{i,k})\to MA_{\omega}(u_{i})$ weakly for $i=1,2$, we get
$$
MA_{\omega}(\phi)\leq \mathbbm{1}_{\{\phi\geq u_{1,j}\}}MA_{\omega}(u_{1})+\mathbbm{1}_{\{\phi\geq u_{2,j}\}}MA_{\omega}(u_{2}).
$$
Therefore letting $j\to \infty$ we obtain
\begin{gather*}
0\leq \int_{X}\big(P_{\omega}(u_{1},u_{2})-\phi\big)MA_{\omega}(\phi)\leq \int_{\{\phi\geq u_{1}\}}\big(P_{\omega}(u_{1},u_{2})-u_{1}\big)MA_{\omega}(u_{1})+\\
+\int_{\{\phi\geq u_{2}\}}\big(P_{\omega}(u_{1},u_{2})-u_{2}\big)MA_{\omega}(u_{2})\leq 0,
\end{gather*}
which by the domination principle of Proposition $3.11.$ in \cite{DDNL17b} implies $P_{\omega}(u_{1},u_{2})\leq \phi$. Hence $P_{\omega}(u_{1},u_{2})=\phi$ which as said above concludes the proof.
\end{proof}
\subsection{The metric space $\big(\bigsqcup_{\psi\in \mathcal{A}}\mathcal{P}(X,\omega,\psi),d_{\mathcal{A}}\big)$.}
\begin{defn}
Let $\psi\in\mathcal{M}$. We introduce the set
$$
\mathcal{P}_{\mathcal{H}}(X,\omega,\psi):=\{P_{\omega}[\psi](\varphi) \, : \, \varphi\in \mathcal{H}\}
$$
where $\mathcal{H}:=\{\varphi\in PSH(X,\omega)\, : \, \omega +dd^{c}\varphi$ Kähler form $\}$.
\end{defn}
Observe that by Lemma \ref{lem:Pro} any $u\in \mathcal{P}_{\mathcal{H}}(X,\omega,\psi)$ has $\psi$-relative minimal singularities. This smaller set is dense in $\big(\mathcal{E}^{1}(X,\omega,\psi),d\big)$ as the next result shows:
\begin{lem}
\label{lem:ProDense}
Let $\psi\in \mathcal{M}$. Then $\mathcal{P}_{\mathcal{H}}(X,\omega,\psi)$ is dense in $\mathcal{E}^{1}(X,\omega,\psi)$ with respect to the strong topology.
\end{lem}
\begin{proof}
We can assume $\psi\in\mathcal{M}^{+}$, otherwise it is trivial. Let $u\in \mathcal{E}^{1}(X,\omega,\psi)$.\\
We first observe that $v_{j}:=P_{\omega}[\psi](\max(u,-j))$ belongs to $\mathcal{E}^{1}(X,\omega,\psi)$ and it has $\psi$-relative minimal singularities by Lemma \ref{lem:Pro}. Moreover since
$$
u=P_{\omega}[\psi](u)\leq v_{j}\leq \max(u,-j)
$$
we also get that $v_{j}\searrow u$. Therefore $d(u,v_{j})\to 0$ for $j\to \infty$ since $d$ is continuous along decreasing sequences. Next, by density of $\mathcal{H}$ into $\mathcal{E}^{1}(X,\omega):=\mathcal{E}^{1}(X,\omega,0)$, for any $j\in \mathbbm{N}$ there exists $\varphi_{j}\in\mathcal{H}$ such that $d(\varphi_{j},\max(u,-j))\leq 1/j$. Therefore, letting $w_{j}:=P_{\omega}[\psi](\varphi_{j})\in \mathcal{P}_{\mathcal{H}}(X,\omega,\psi)$, by Proposition \ref{prop:Dist} it follows that
$$
d(w_{j},u)\leq d(w_{j},v_{j})+d(v_{j},u)\leq d(u,v_{j})+\frac{1}{j}\to 0
$$
as $j\to\infty$, which concludes the proof.
\end{proof}
\begin{rem}
\emph{By the main Theorem in \cite{Mc18}, if the singularities of $\psi$ are} analytic, \emph{i.e. $\psi=P_{\omega}[u]$ for $u$ with analytic singularities of type $\mathfrak{a}^{c}$ for an analytic coherent ideal sheaf $\mathfrak{a}\subset \mathcal{O}_{X}$, $c\in\mathbbm{R}_{>0}$, then any function $v\in \mathcal{P}_{\mathcal{H}}(X,\omega,\psi)$ is $\mathcal{C}^{1,1}_{loc}\big(X\setminus V(\mathfrak{a})\big)$.}
\end{rem}
We need now to recall the definition of the \emph{entropy}.
\begin{defn}\emph{[Definition 2.9., \cite{BBEGZ16}]}
Let $\mu,\nu$ two probability measures on $X$. The \emph{relative entropy} $H_{\mu}(\nu)\in [0,+\infty]$ of $\nu$ with respect to $\mu$ is defined as follows. If $\nu$ is absolutely continuous with respect to $\mu$ and $f:=\frac{d\nu}{d\mu}$ satisfies $f\log f\in L^{1}(\mu)$ then
$$
H_{\mu}(\nu):=\int_{X}f\log f d\mu=\int_{X}\log \Big(\frac{d\nu}{d\mu}\Big)d\nu.
$$
Otherwise $H_{\mu}(\nu):=+\infty$.
\end{defn}
The relative entropy provides compact sets in $\mathcal{E}^{1}(X,\omega)$ endowed with the strong topology (Definition \ref{defn:Strong}).
\begin{thm}\emph{[Theorem $2.17.$, \cite{BBEGZ16}]}
\label{thm:Comp}
Let $C$ be a positive constant. Then the set
$$
\mathcal{K}_{C}:=\Big\{\varphi\in \mathcal{E}^{1}(X,\omega)\, :\, \max \Big(|\sup_{X}\varphi|, H_{MA_{\omega}(0)/V_{0}}\big(MA_{\omega}(\varphi)/V_{0}\big)\Big)\leq C\Big\}
$$
is compact in $\mathcal{E}^{1}(X,\omega)$ with respect to the strong topology.
\end{thm}
\begin{defn}
Given $\psi\in \mathcal{A}$, we define for any $C\geq 0$
\begin{multline*}
\mathcal{P}_{C}(X,\omega,\psi):=P_{\omega}[\psi](K_{C})=\Big\{P_{\omega}[\psi](\varphi)\in \mathcal{E}^{1}(X,\omega,\psi) \, :\, \max\big(|\sup_{X}\varphi|, H_{MA_{\omega}(0)}(MA_{\omega}(\varphi))\big)\leq C\Big\},
\end{multline*}
and
$$
\mathcal{P}(X,\omega,\psi):=\bigcup_{C\in \mathbbm{R}_{\geq 0}}\mathcal{P}_{C}(X,\omega,\psi).
$$
\end{defn}
As an immediate consequence of Theorem \ref{thm:Comp} and Proposition \ref{prop:UC} $\mathcal{P}_{C}(X,\omega,\psi)$ is compact in $(\mathcal{E}^{1}(X,\omega,\psi),d)$ and $P_{C_{1}}(X,\omega,\psi)\subset \mathcal{P}_{C_{2}}(X,\omega,\psi)$ if $C_{1}\leq C_{2}$. Moreover since $ H_{MA_{\omega}(0)}(MA_{\omega}(\varphi))<\infty$ for any $\varphi\in \mathcal{H}$, $\mathcal{P}_{\mathcal{H}}(X,\omega,\psi)=\bigcup_{C\in\mathbbm{R}_{\geq 0}}P_{\omega}[\psi](\mathcal{K}_{C}\cap\mathcal{H})\subset \mathcal{P}(X,\omega,\psi)$. It is also clear that for any $u\in\mathcal{P}(X,\omega,\psi)$ there exists $C\in\mathbbm{R}_{\geq 0}$ minimal such that $u\in\mathcal{P}_{C}(X,\omega,\psi)$. We set $\mathcal{P}(X,\omega):=\mathcal{P}(X,\omega,0)$ and we call $\varphi\in\mathcal{P}(X,\omega)$ a \emph{minimal entropy function} for $u\in\mathcal{P}(X,\omega,\psi)$ if $\varphi\in \mathcal{K}_{C}$ and $P_{\omega}[\psi](\varphi)=u$ for $C$ minimal.
\begin{defn}
\label{defn:Ref}
Let $u\in \mathcal{P}_{C_{1}}(X,\omega,\psi_{1}), v\in \mathcal{P}_{C_{2}}(X,\omega,\psi_{2})$ for $\psi_{1},\psi_{2}\in \mathcal{A}$ such that $\psi_{2}\preccurlyeq \psi_{1}$. Assume also that $C_{1}$ (respectively $C_{2}$) is minimal such that $u\in\mathcal{P}_{C_{1}}(X,\omega,\psi_{1})$ (resp. $v\in\mathcal{P}_{C_{2}}(X,\omega,\psi_{2})$). We define
$$
\tilde{d}_{\mathcal{A}}(u,v):=d(v,P_{\omega}[\psi_{2}](u))+\sup \Big\{d(a,b)-d(P_{\omega}[\psi_{2}](a),P_{\omega}[\psi_{2}](b))\Big\}+V_{u}-V_{v}
$$
where the supremum is over all $a,b\in \mathcal{P}_{\max(C_{1},C_{2})}(X,\omega,\psi_{1})$.
\end{defn}
We observe that $\tilde{d}_{\mathcal{A}}$ takes finite values since the supremum in the definition is actually equal to
$$
\max_{(\varphi_{1},\varphi_{2})\in \mathcal{K}_{\max(C_{1},C_{2})}\times \mathcal{K}_{\max(C_{1},C_{2})}}\Big\{d\big(P_{\omega}[\psi_{1}](\varphi_{1}),P_{\omega}[\psi_{1}](\varphi_{2})\big)-d\big(P_{\omega}[\psi_{2}](\varphi_{1}),P_{\omega}[\psi_{2}](\varphi_{2})\big)\Big\}.
$$
\begin{prop}
\label{prop:Proper}
Let $u\in \mathcal{P}(X,\omega,\psi_{1}), v\in \mathcal{P}(X,\omega,\psi_{2})$ for $\psi_{1},\psi_{2}\in \mathcal{A}$ such that $\psi_{2}\preccurlyeq \psi_{1}$. Then the followings hold:
\begin{itemize}
\item[i)] $\tilde{d}_{\mathcal{A}}(u,v)=\tilde{d}_{\mathcal{A}}(v,u)$;
\item[ii)] $\tilde{d}_{\mathcal{A}}(u,v)\in \mathbbm{R}_{\geq 0}$ and $\tilde{d}_{\mathcal{A}}(u,v)=0$ if and only if $u=v$;
\item[iii)] if $\psi_{1}=\psi_{2}$ then $\tilde{d}_{\mathcal{A}}(u,v)=d(u,v)$;
\item[iv)] $\tilde{d}_{\mathcal{A}}(u,v)\geq d(v,P_{\omega}[\psi_{2}](u))$ and $\tilde{d}_{\mathcal{A}}(u,v)\geq d(u, P_{\omega}[\psi_{1}](\varphi)) $ where $\varphi\in\mathcal{P}(X,\omega)$ is a minimal entropy function for $v$.
\end{itemize}
\end{prop}
\begin{proof}
The first point is trivial. By Proposition \ref{prop:Dist} and the main Theorem in \cite{WN17} $\tilde{d}_{\mathcal{A}}\in\mathbbm{R}_{\geq 0}$, and if $\psi_{1}=\psi_{2}$ then $\tilde{d}_{\mathcal{A}}(u,v)=d(u,v)$. For $(iv)$, instead, the first inequality is immediate, while the second inequality follows considering $a=u, b=P_{\omega}[\psi_{1}](\varphi)$ in the supremum.\\
Therefore it remains to prove that $\tilde{d}_{\mathcal{A}}(u,v)=0$ implies $u=v$. But if $\tilde{d}_{\mathcal{A}}(u,v)=0$ then in particular $V_{\psi_{1}}=V_{\psi_{2}}$, and Theorem \ref{thm:samesing} implies $\psi_{1}=\psi_{2}$. Hence the third point and Theorem \ref{thmA} conclude the proof.
\end{proof}
The map $\tilde{d}_{\mathcal{A}}$ does not seem to be a distance on $\bigsqcup_{\psi\in \mathcal{A}}\mathcal{P}(X,\omega,\psi)$, since it hardly satisfies the triangle inequality. Indeed it is composed of three parts, and clearly two parts behaves well for the triangle inequality, but the part given by the supremum seem to be very unstable since the set of the supremum depends on the functions $u,v$ taken. Therefore we want to modify $\tilde{d}_{\mathcal{A}}$ to get a distance $d_{\mathcal{A}}$ which still coincides with the $d$-distance on $\mathcal{P}(X,\omega,\psi)$ for any $\psi\in \mathcal{A}$. The next Lemma is the key point to proceed.
\begin{lem}
\label{lem:Keyy}
Let $u,v\in \mathcal{P}(X,\omega,\psi)$ for $\psi\in \mathcal{A}$. Then for any $m\in \mathbbm{N}$ and any $w_{1},\dots,w_{m}\in \bigsqcup_{\psi'\in \mathcal{A}}\mathcal{P}(X,\omega,\psi')$,
$$
d(u,v)\leq \tilde{d}_{\mathcal{A}}(u,w_{1})+\sum_{j=1}^{m-1}\tilde{d}_{\mathcal{A}}(w_{j},w_{j+1})+\tilde{d}_{\mathcal{A}}(w_{m},v).
$$
\end{lem}
The proof of this Lemma is quite laborious and it will be presented in the subsection \ref{ssec:Lem}.\\

Next we define $d_{\mathcal{A}}:\bigsqcup_{\psi\in \mathcal{A}}\mathcal{P}(X,\omega,\psi)\times \bigsqcup_{\psi\in \mathcal{A}}\mathcal{P}(X,\omega,\psi)\to \mathbbm{R}_{\geq 0}$ as
$$
d_{\mathcal{A}}(u,v):=\inf \Big\{\tilde{d}_{\mathcal{A}}(u,w_{1})+\sum_{j=1}^{m-1}\tilde{d}_{\mathcal{A}}(w_{j},w_{j+1})+\tilde{d}_{\mathcal{A}}(w_{m},v)\Big\}
$$
where the infimum is over all possible chains in $\bigsqcup_{\psi\in \mathcal{A}}\mathcal{P}(X,\omega,\psi)$.\\
We can now prove Theorem \ref{thmC}:
\begin{reptheorem}{thmC}
$\Big(\bigsqcup_{\psi\in \mathcal{A}}\mathcal{P}(X,\omega,\psi),d_{\mathcal{A}}\Big)$ is a metric space, and denoting with $X_{\mathcal{A}}$ its metric completion, we have
$$
X_{\mathcal{A}}=\bigsqcup_{\psi\in \mathcal{\overline{A}}}\mathcal{E}^{1}(X,\omega,\psi)
$$
where $\mathcal{\overline{A}}\subset \mathcal{M}$ is the closure of $\mathcal{A}$ as subset of $PSH(X,\omega)$ with its $L^{1}$-topology and where we identify $\mathcal{E}^{1}(X,\omega,\psi')$ with a singleton $P_{\psi'}$ if $\psi':=\inf\mathcal{A}\in\mathcal{M}\setminus \mathcal{M}^{+}$.\\
In particular the complete metric space $(X_{\mathcal{A}},d_{\mathcal{A}})$ restricts to $\big(\mathcal{E}^{1}(X,\omega,\psi),d\big)$ on $\mathcal{E}^{1}(X,\omega,\psi)$ for any $\psi\in\mathcal{\overline{A}}$. 
\end{reptheorem}
\begin{proof}
\textbf{Step 1: $\Big(\bigsqcup_{\psi\in\mathcal{A}}\mathcal{P}(X,\omega,\psi),d_{\mathcal{A}}\Big)$ is a metric space}.\\
As a consequence of Lemma \ref{lem:Keyy} we immediately get
$$
d_{\mathcal{A}|\mathcal{P}(X,\omega,\psi)\times \mathcal{P}(X,\omega,\psi)}=d
$$
for any $\psi\in \mathcal{A}$. Therefore to prove that $d_{\mathcal{A}}$ is a distance on $\bigsqcup_{\psi\in\mathcal{A}}\mathcal{P}(X,\omega,\psi)$ it remains to prove that $d_{\mathcal{A}}(u,v)=0$ implies $u=v$ since the triangle inequality easily follows from the construction (see also Proposition \ref{prop:Proper}). But given $w_{1},\dots,w_{m}\in\bigsqcup_{\psi\in \mathcal{A}}\mathcal{P}(X,\omega,\psi)$, the uniform bound
$$
\tilde{d}_{\mathcal{A}}(u,w_{1})+\tilde{d}_{\mathcal{A}}(w_{1},w_{2})+\cdots+\tilde{d}_{\mathcal{A}}(w_{m},v)\geq |V_{u}-V_{v}|
$$ 
holds. Therefore $d_{\mathcal{A}}(u,v)=0$ leads to $V_{u}=V_{v}$, and since $\mathcal{A}$ is totally ordered, by Theorem \ref{thm:samesing}, we obtain that $u,v\in \mathcal{E}^{1}(X,\omega,\psi)$ for a common $\psi\in \mathcal{A}$. Hence $0=d_{\mathcal{A}}(u,v)=d(u,v)$, which implies $u=v$ and concludes the first step.\\
\textbf{Step 2: $\Big(\bigsqcup_{\psi\in \mathcal{\overline{A}}} \big(\mathcal{E}^{1}(X,\omega,\psi),d\big) \Big)\subset (X_{\mathcal{A}},d_{\mathcal{A}})$.}\\
For any $\psi\in (\mathcal{\overline{A}}\setminus \mathcal{A})$ there exists a monotone sequence $\{\psi_{k}\}_{k\in\mathbbm{N}}$ such that $\psi=\big(\lim_{k\to \infty}\psi_{k}\big)^{*}$. Thus, letting $\mathcal{P}_{C}(X,\omega,\psi)\ni u=P_{\omega}[\psi](\varphi)$ for $\varphi\in \mathcal{K}_{C}$ minimal entropy function for $u$ and letting $u_{k}:=P_{\omega}[\psi_{k}](\varphi)$, we claim that $\{u_{k}\}_{k\in\mathbbm{N}}$ is a Cauchy sequence with respect the distance $d_{\mathcal{A}}$.\\
Indeed if $\{\psi_{k}\}_{k\in\mathbbm{N}}$ is increasing, then for any $j,k$ such that $j\geq k$ we have
\begin{multline*}
d_{\mathcal{A}}(u_{k},u_{j})\leq \tilde{d}_{\mathcal{A}}(u_{k},u_{j})\leq\sup_{a,b\in \mathcal{P}_{C}(X,\omega,\psi_{j})}\Big\{d(a,b)-d\big(P_{\omega}[\psi_{k}](a),P_{\omega}[\psi_{k}](b)\big)\Big\}+V_{u_{j}}-V_{u_{k}}\leq\\
\leq \sup_{a,b\in \mathcal{P}_{C}(X,\omega,\psi)}\Big\{d(a,b)-d\big(P_{\omega}[\psi_{k}](a),P_{\omega}[\psi_{k}](b)\big)\Big\}+V_{\psi}-V_{\psi_{k}}
\end{multline*}
by the definition of $d_{\mathcal{A}}$ and Proposition \ref{prop:Dist} since $\psi \succcurlyeq \psi_{k}$ for any $k\in\mathbbm{N}$. Therefore by Proposition \ref{prop:UC} (see also Lemma \ref{lem:Volumes1}) $\{u_{k}\}_{k\in\mathbbm{N}}$ is a Cauchy sequence in $\Big(\bigsqcup_{\psi\in \mathcal{A}}\mathcal{P}(X,\omega,\psi),d_{\mathcal{A}}\Big)$.\\
If instead $\psi_{k}\searrow \psi$, we first denote by $C_{1}\in\mathbbm{R}_{\geq 0}$ the minimal constant such that $u_{1}=P_{\omega}[\psi_{1}](\phi)$ for $\phi\in \mathcal{K}_{C_{1}}$. Thus for any $j,k$ such that $j\geq k$ we have
\begin{multline*}
d_{\mathcal{A}}(u_{k},u_{j})\leq \tilde{d}_{\mathcal{A}}(u_{k},u_{j})\leq\sup_{a,b\in \mathcal{P}_{C_{1}}(X,\omega,\psi_{k})}\Big\{d(a,b)-d\big(P_{\omega}[\psi_{j}](a),P_{\omega}[\psi_{j}](b)\big)\Big\}+V_{u_{k}}-V_{u_{j}}\leq\\
\leq \sup_{a,b\in \mathcal{P}_{C_{1}}(X,\omega,\psi_{k})}\Big\{d(a,b)-d\big(P_{\omega}[\psi](a),P_{\omega}[\psi](b)\big)\Big\}+V_{\psi_{k}}-V_{\psi},
\end{multline*}
and as before Proposition \ref{prop:UC} and Lemma \ref{lem:Volumes1} imply that $\{u_{k}\}_{k\in\mathbbm{N}}$ is a Cauchy sequence.\\
Hence we define the map
$$
\tilde{\Phi}: \Big(\bigsqcup_{\psi\in \mathcal{\overline{A}}} \big(\mathcal{P}(X,\omega,\psi),d\big) \Big)\to (X_{\mathcal{A}},d_{\mathcal{A}})
$$
as $\tilde{\Phi}(u):=[u_{k}]$ where we recall that $(X_{\mathcal{A}},d_{\mathcal{A}})$ is the metric completion of $\bigsqcup_{\psi\in\mathcal{A}}\big(\mathcal{P}(X,\omega,\psi),d\big)$. We need to check that it is well-defined.\\
Let assume $u=P_{\omega}[\psi](\varphi)=P_{\omega}[\psi](\varphi')$ for $\varphi,\varphi'\in \mathcal{K}_{C}$ minimal entropy functions. Define also $u_{k}:=P_{\omega}[\psi_{k}](\varphi)$, $u_{k}':=P_{\omega}[\psi_{k}](\varphi')$ where $\psi=\big(\lim_{k\to \infty}\psi_{k}\big)^{*}$ monotonically. Then by Proposition \ref{prop:UC}
$$
\limsup_{k\to\infty} d_{\mathcal{A}}(u_{k},u_{k}')=\lim_{k\to\infty}d(u_{k},u_{k}')=d(u,u)=0.
$$
Next assume that $u=P_{\omega}[\psi](\varphi)$, $u_{k}:=P_{\omega}[\psi_{k}](\varphi)$, $u_{k}':=P_{\omega}[\psi_{k}'](\varphi)$ for $\{\psi_{k}\}_{k\in\mathbbm{N}}, \{\psi_{k}'\}_{k\in\mathbbm{N}}\subset \mathcal{A}$ monotone sequence converging to $\psi$ almost everywhere. We need to check that $d_{\mathcal{A}}(u_{k},u_{k}')\to 0$ as $k\to \infty$. Let $C_{1},C_{1}'\in\mathbbm{R}$ be minimal constants such that $u_{1}=P_{\omega}[\psi_{1}](\phi), u_{1}=P_{\omega}[\psi_{1}'](\phi')$ for $\phi\in \mathcal{K}_{C_{1}}, \phi'\in\mathcal{K}_{C_{1}'}$, and set $C_{2}:=\max(C,C_{1},C_{1}')$. Then for any $k\in\mathbbm{N}$ such that $\psi_{k}'\preccurlyeq \psi_{k}$,
$$
d_{\mathcal{A}}(u_{k},u_{k}')\leq \tilde{d}_{\mathcal{A}}(u_{k},u_{k}') \leq \sup_{(a,b)\in \mathcal{P}_{C_{2}}(X,\omega,\psi_{k})}\Big\{d(a,b)-d\big(P_{\omega}[\psi_{k}'](a),P_{\omega}[\psi_{k}'](b)\big)\Big\}+V_{\psi_{k}}-V_{\psi_{k}'}
$$
and similarly if $\psi_{k}'\succcurlyeq \psi_{k}$. Therefore, proceeding similarly as before, it is not difficult to check that $d_{\mathcal{A}}(u_{k},u_{k}')\to 0$ as $k\to \infty$ using again Proposition \ref{prop:UC}. Hence $\tilde{\Phi}$ is well-defined.\newline
We need also to check that $\tilde{\Phi}$ is injective. Fix $u,v\in \bigsqcup_{\psi\in \mathcal{\overline{A}}} \mathcal{P}(X,\omega,\psi)$ such that $u\in\mathcal{P}(X,\omega,\psi_{1}),v\in\mathcal{P}(X,\omega,\psi_{2})$ for $\psi_{1},\psi_{2}\in \overline{\mathcal{A}}$, and choose $\{\psi_{1,k},\psi_{2,k}\}_{k\in\mathbbm{N}}\subset \mathcal{A}$ sequences such that $\psi_{1,k}\to \psi_{1}, \psi_{2,k}\to \psi_{2}$ monotonically. Then, by definition of metric completion and of the metric $d_{\mathcal{A}}$, the corresponding points $\tilde{\Phi}(u)=[u_{k}], \tilde{\Phi}(v)=[v_{k}]$ satisfies
$$
d_{\mathcal{A}}\big([u_{k}],[v_{k}]\big):=\lim_{k\to \infty}d_{\mathcal{A}}(u_{k},v_{k})\geq \lim_{k\to \infty} |V_{\psi_{1,k}}-V_{\psi_{2,k}}|=|V_{\psi_{1}}-V_{\psi_{2}}|.
$$
In particular, if $d_{\mathcal{A}}\big([u_{k}],[v_{k}]\big)=0$ then necessarily $\psi_{1}=\psi_{2}$ (see also Lemma \ref{lem:Volumes1}). But if $u:=P_{\omega}[\psi](\varphi_{1}),v:=P_{\omega}[\psi](\varphi_{2})$ for minimal entropy functions $\varphi_{1},\varphi_{2}$ then
\begin{equation}
\label{eqn:Final?}
\lim_{k\to \infty}d_{\mathcal{A}}(u_{k},v_{k})=\lim_{k\to \infty}d(u_{k},v_{k})=d(u,v)
\end{equation}
where $u_{k}:=P_{\omega}[\psi_{k}](\varphi_{1})$, $v_{k}:=P_{\omega}[\psi_{k}](\varphi_{2})$ for $\{\psi_{k}\}_{k\in\mathbbm{N}}\subset \mathcal{A}$ monotonically converging a.e. to $\psi$, and where the last equality again follows from Proposition \ref{prop:UC}. Thus if $\lim_{k\to\infty}d_{\mathcal{A}}(u_{k},v_{k})=0$ then $u=v$ (recall that in the case $\psi=\inf \mathcal{A}\in\mathcal{M}\setminus\mathcal{M}^{+}$ we identify $u,v$ with the singleton $P_{\psi}$), and the injectivity of $\tilde{\Phi}$ follows.\newline
Finally note also that $(\ref{eqn:Final?})$ implies that $\tilde{\Phi}$ is an isometric embedding, which admits a unique continuous extension
$$
\tilde{\Phi}: \Big(\bigsqcup_{\psi\in \mathcal{\overline{A}}} \big(\mathcal{E}^{1}(X,\omega,\psi),d\big) \Big)\to (X_{\mathcal{A}},d_{\mathcal{A}})
$$
thanks to Lemma \ref{lem:ProDense}.\newline
\textbf{Step 3: set up the final strategy.}\\
It remains to prove that $\bigsqcup_{\psi\in\mathcal{\overline{A}}}\big(\mathcal{E}^{1}(X,\omega,\psi),d_{\mathcal{A}}\big)$ is complete. Thus let $\{u_{j}\}_{j\in\mathbbm{N}}\subset \bigsqcup_{\psi\in\mathcal{\overline{A}}}\mathcal{P}(X,\omega,\psi)$ be a Cauchy sequence. Up to extract a subsequence, we may assume $d_{\mathcal{A}}(u_{j},u_{j+1})\leq \frac{1}{2^{j}}$. For any $j\in\mathbbm{N}$ let also $\varphi_{j}\in\mathcal{P}(X,\omega)$ a minimal entropy function for $u_{j}$ and $\psi_{j}\in\mathcal{\overline{A}}$ such that $u_{j}\in\mathcal{P}(X,\omega,\psi_{j})$. Since $\mathcal{\overline{A}}$ is totally ordered, up to consider a subsequence, we may assume that $\{\psi_{j}\}_{j\in\mathbbm{N}}$ converges monotonically a.e. to $\psi\in \mathcal{\overline{A}}$.\\
\textbf{Step 4: $\{\psi_{j}\}_{j\in\mathbbm{N}}$ increasing.}\\
Let for any $k\geq j$, $v_{j,k}:=P_{\omega}\Big(P_{\omega}[\psi_{j}](\varphi_{j}),\cdots,P_{\omega}[\psi_{j}](\varphi_{k})\Big)$ and let for any $k\geq j$, $i\in\mathbbm{N}$, $v_{j,k}^{i}:=P_{\omega}\Big(P_{\omega}[\psi_{i}](\varphi_{j}),\cdots,P_{\omega}[\psi_{i}](\varphi_{k})\Big)$. Note that $v_{j,k}^{j}=v_{j,k}$ and that $v_{j,k}=P_{\omega}(u_{j},v_{j+1,k}^{j})$. Moreover we claim that $P_{\omega}[\psi_{j}](v_{j,k}^{i})=v_{j,k}$ if $i\geq j$. Indeed $P_{\omega}[\psi_{j}](v_{j,k}^{i})\leq v_{j,k}$ since $v_{j,k}^{i}\leq P_{\omega}[\psi_{i}](\varphi_{j}),\dots, P_{\omega}[\psi_{i}](\varphi_{k})$ implies $P_{\omega}[\psi_{j}](v_{j,k}^{i})\leq P_{\omega}[\psi_{j}](\varphi_{j}),\dots,P_{\omega}[\psi_{j}](\varphi_{k})$ by Lemma \ref{lem:Pro}. While the reverse inequality follows applying $P_{\omega}[\psi_{j}](\cdot)$ to the trivial inequality $v_{j,k}\leq v_{j,k}^{i}$. As a consequence we get that
$$
d(u_{j},v_{j,k})= d\big(u_{j},P_{\omega}(u_{j},v_{j+1,k}^{j})\big)\leq d(u_{j}, v_{j+1,k}^{j})\leq d_{\mathcal{A}}(u_{j},v_{j+1,k})
$$
where the last inequality follows from Proposition \ref{prop:Proper}.$(iv)$. Iterating, by the triangle inequality we have
$$
d(u_{j},v_{j,k})\leq \sum_{l=0}^{k-j-1}d(u_{j+l},u_{j+l+1})\leq \sum_{l=0}^{k-j-l}\frac{1}{2^{j+l}}\leq \frac{1}{2^{j-1}}.
$$
Clearly $v_{j,k}$ is decreasing in $k$, thus, letting $C_{j}\in\mathbbm{N}$ such that $v_{j,k}\leq \psi_{j}+C_{j}$ for any $k\in\mathbbm{N}$, we get
$$
C_{j}-E_{\psi_{j}}(v_{j,k})=E_{\psi_{j}}(\psi_{j}+C_{j})-E_{\psi_{j}}(v_{j,k})=d(\psi_{j}+C_{j},v_{j,k})\leq d(\psi_{j}+C_{j},u_{j})+\frac{1}{2^{j-1}},
$$
which implies that $v_{j}:=\lim_{k\to \infty}v_{j,k}\in\mathcal{E}^{1}(X,\omega,\psi_{j})$ by Proposition \ref{prop:Decr}. Moreover $d(u_{j},v_{j})\leq 2^{-j+1}$ by continuity along decreasing sequence. Observe also that $v_{j}\leq v_{j+1}$ by construction since $v_{j}\leq v_{j,k}\leq v_{j,k}^{j+1}\leq v_{j+1,k}$ for any $k\geq j+1$. \\
Then by Lemma \ref{lem:Bound} there exists two uniform constants $A>1,B>0$ such that
$$
\sup_{X}(v_{j}-\psi_{j})=\sup_{X}v_{j}\leq \frac{1}{V_{\psi_{j}}}\big(Ad(\psi_{j},v_{j})+B\big)
$$
which implies that $u:=\big(\lim_{j\to \infty}v_{j}\big)^{*}\in PSH(X,\omega,\psi)$. Therefore, assuming $\sup_{X}u=0$ up to add a constant, by Theorem \ref{thm:conv} we also have $MA_{\omega}(v_{j})\to MA_{\omega}(u)$ weakly, which implies $V_{u}=V_{\psi}$ and, for any $m\in\mathbbm{N}$ fixed,
\begin{multline*}
\int_{X}\big(\psi-\max(u,\psi-m)\big)MA_{\omega}\big(\max(u,\psi-m)\big)=\lim_{j\to \infty}\int_{X}\big(\psi_{j}-\max(v_{j},\psi_{j}-m)\big)MA_{\omega}\big(\max(v_{j},\psi_{j}-m)\big)\leq\\
\leq \lim_{j\to \infty} (n+1)d\big(\psi_{j},\max(v_{j},\psi_{j}-m)\big)\leq \lim_{j\to \infty}(n+1)d(\psi_{j},v_{j})
\end{multline*}
using also that $\max(v_{j},\psi_{j}-m)\nearrow \max(u,\psi-m)$ almost everywhere. Therefore $u\in\mathcal{E}^{1}(X,\omega,\psi)$ as a consequence of
$$
d(\psi_{j},v_{j})\leq d_{\mathcal{A}}(\psi_{j},\psi_{1})+d(\psi_{1},u_{1})+d_{\mathcal{A}}(u_{1},u_{j})+d(u_{j},v_{j})\leq V_{\psi}-V_{\psi_{1}}+d(\psi_{1},u_{1})+2.
$$
Thus to finish this step it remains to check that $d_{\mathcal{A}}(u,u_{j})\to 0$ as $j\to \infty$, or equivalently that $d_{\mathcal{A}}(u,v_{j})\to 0$.\\
Set for any $k\geq j$, $v_{k}^{j}:=P_{\omega}[\psi_{j}](v_{k})$. Then by construction $\{v_{k}^{j}\}_{k\geq j}$ is an increasing sequence converging strongly in $\mathcal{E}^{1}(X,\omega,\psi_{j})$ to $P_{\omega}[\psi_{j}](u)$. For $\epsilon>0$ fixed, let also $\phi_{\epsilon}\in\mathcal{H}$ such that $d\big(u,P_{\omega}[\psi](\phi_{\epsilon})\big)\leq \epsilon$. Next, for any $j$ fixed let $s\in\mathbbm{N}$ depending on $j,\epsilon$ such that $d\big(v_{j+s}^{j},P_{\omega}[\psi_{j}](u)\big)\leq \epsilon$. Thus by the triangle inequality and Proposition \ref{prop:Dist} we have
\begin{gather*}
d_{\mathcal{A}}(v_{j},u)\leq \sum_{l=0}^{s-1}d(v_{j+l}^{j},v_{j+l+1}^{j})+d\big(v_{j+s}^{j},P_{\omega}[\psi_{j}](u)\big)+d\big(P_{\omega}[\psi_{j}](u),P_{\omega}[\psi_{j}](\phi_{\epsilon})\big)+\\
+d_{\mathcal{A}}\big(P_{\omega}[\psi_{j}](\phi_{\epsilon}),P_{\omega}[\psi](\phi_{\epsilon})\big)+d\big(P_{\omega}[\psi](\phi_{\epsilon}),u\big)\leq \frac{1}{2^{j-2}}+3\epsilon+d_{\mathcal{A}}\big(P_{\omega}[\psi_{j}](\phi_{\epsilon}),P_{\omega}[\psi](\phi_{\epsilon})\big),
\end{gather*}
which by Proposition \ref{prop:UC} implies $\limsup_{j\to \infty}d_{\mathcal{A}}(v_{j},u)\leq 3\epsilon$ since
\begin{multline*}
d_{\mathcal{A}}\big(P_{\omega}[\psi_{j}](\phi_{\epsilon}),P_{\omega}[\psi](\phi_{\epsilon})\big)\leq \tilde{d}_{\mathcal{A}}\big(P_{\omega}[\psi_{j}](\phi_{\epsilon}),P_{\omega}[\psi](\phi_{\epsilon})\big)\leq\\
\leq\sup_{\varphi_{1},\varphi_{2}\in \mathcal{P}_{C_{\epsilon}}(X,\omega,\psi)}\{d(\varphi_{1},\varphi_{2})-d\big(P_{\omega}[\psi_{j}](\varphi_{1}),P_{\omega}[\psi_{j}](\varphi_{2})\big)\}+V_{\psi}-V_{\psi_{j}}\to 0
\end{multline*}
for a certain constant $C_{\epsilon}\in\mathbbm{R}$.\\
\textbf{Step 4: $\{\psi_{j}\}_{j\in\mathbbm{N}}$ decreasing.}\\
We define for any $j\in\mathbbm{N}$, $w_{j}:=P_{\omega}[\psi](u_{j})$. Clearly $\{w_{j}\}_{j\in\mathbbm{N}}\subset \mathcal{E}^{1}(X,\omega,\psi)$ is a Cauchy sequence since for any $j\geq k$
$$
d(w_{j},w_{k})\leq d\big(u_{j},P_{\omega}[\psi_{j}](u_{k})\big)\leq d_{\mathcal{A}}(u_{j},u_{k})\leq \frac{1}{2^{k-1}}
$$
by Proposition \ref{prop:Dist} and Proposition \ref{prop:Proper}. Thus $w_{j}$ converges strongly to a function $u$ in $\mathcal{E}^{1}(X,\omega,\psi)$ and to conclude the proof it remains to prove that $d_{\mathcal{A}}(u_{j},u)\to 0$ as $j\to \infty$. Therefore, letting for any $k\in\mathbbm{N}$, $\phi_{k}\in\mathcal{H}$ such that $d\big(u_{k},P_{\omega}[\psi_{k}](\phi_{k})\big)\leq 1/k$, we get for any $k\leq j$
\begin{multline}
\label{eqn:Con}
d_{\mathcal{A}}(u_{j},u)\leq d\big(u_{j},P_{\omega}[\psi_{j}](u_{k})\big)+d\big(P_{\omega}[\psi_{j}](u_{k}),P_{\omega}[\psi_{j}](\phi_{k})\big)+d_{\mathcal{A}}\big(P_{\omega}[\psi_{j}](\phi_{k}),P_{\omega}[\psi](\phi_{k})\big)+\\
+d\big(P_{\omega}[\psi](\phi_{k}),u\big)\leq d_{\mathcal{A}}(u_{j},u_{k})+d\big(u_{k},P_{\omega}[\psi_{k}](\phi_{k})\big)+d_{\mathcal{A}}\big(P_{\omega}[\psi_{j}](\phi_{k}),P_{\omega}[\psi](\phi_{k})\big)+\\
+d\big(P_{\omega}[\psi](\phi_{k}),u\big)\leq \frac{1}{2^{k-1}}+\frac{1}{k}+d_{\mathcal{A}}\big(P_{\omega}[\psi_{j}](\phi_{k}),P_{\omega}[\psi](\phi_{k})\big)+d\big(P_{\omega}[\psi](\phi_{k}),u\big)
\end{multline}
combining Proposition \ref{prop:Dist} and Proposition \ref{prop:Proper}. Therefore since clearly $P_{\omega}[\psi](\phi_{k})$ converges strongly to $u$ in $\mathcal{E}^{1}(X,\omega,\psi)$ and as before $\limsup_{j\to \infty}d_{\mathcal{A}}\big(P_{\omega}[\psi_{j}](\phi_{k}),P_{\omega}[\psi](\phi_{k})\big)=0$, it follows from (\ref{eqn:Con}) that $\limsup_{j\to \infty}d_{\mathcal{A}}(u_{j},u)=0$ letting $j\to \infty$ and then $k\to \infty$.
\end{proof}
\subsection{Proof of Lemma \ref{lem:Keyy}.}
\label{ssec:Lem}
The proof of Lemma \ref{lem:Keyy} proceeds by induction on $m\in\mathbbm{N} $ length of the chain.\\
\textbf{Step 1 ($\mathbf{m=1}$):} Assume $w\in \mathcal{P}(X,\omega,\psi')$ for $\psi'\in \mathcal{A}$. Then by Proposition \ref{prop:Proper}$.(iv)$ we get that
$$
\tilde{d}_{\mathcal{A}}(u,w)+\tilde{d}_{\mathcal{A}}(w,v)\geq d(u, P_{\omega}[\psi](\varphi))+d(P_{\omega}[\psi](\varphi),v)\geq d(u,v)
$$
where $\varphi\in \mathcal{P}(X,\omega)$ is a minimal entropy function for $w$.\\
\textbf{Step 2 ($\mathbf{m\to m+1}$): reduce to an easier case 1.} Assume now that the Lemma holds for any chain of length $n\leq m\in\mathbbm{N}$, and let $w_{1},\dots,w_{m+1}\in\bigsqcup_{\psi'\in \mathcal{A}}\mathcal{P}(X,\omega,\psi')$. To fix the notations assume $w_{j}\in \mathcal{P}(X,\omega,\psi_{j})$ and that, for any $j=0,\dots,n$, $\varphi_{j}\in \mathcal{K}_{C_{j}}$ is a choice of minimal entropy functions for $w_{j}$.\\
Next, using the definition of $\tilde{d}_{\mathcal{A}}$ and Proposition \ref{prop:Proper}, if $\psi_{j+1}\preccurlyeq \psi_{j-1}\preccurlyeq \psi_{j}$ then
$$
\tilde{d}_{\mathcal{A}}(w_{j-1},w_{j})+\tilde{d}_{\mathcal{A}}(w_{j},w_{j+1})\geq \tilde{d}_{\mathcal{A}}(w_{j-1},P_{\omega}[\psi_{j-1}](\varphi_{j}))+\tilde{d}_{\mathcal{A}}(P_{\omega}[\psi_{j-1}](\varphi_{j}),w_{j+1}).
$$
Therefore we may assume there exists $j_{0}\in\{1,\dots,m+1\}$ such that $\psi\succcurlyeq \psi_{1}\succcurlyeq \cdots\succcurlyeq \psi_{j_{0}}$ and $\psi_{j_{0}}\preccurlyeq \psi_{j_{0}+1}\preccurlyeq \cdots\preccurlyeq \psi$.\\
\textbf{Step 3 ($\mathbf{m\to m+1}$): reduce to an easier case 2.} We claim that if there exists $j\in\{1,\dots,m+1\}$ such that $C_{j}\geq \max(C_{j-1},C_{j+1})$ (where we set $w_{0}:=u$, $w_{m+2}:=v$) then
\begin{equation}
\label{eqn:Three}
\tilde{d}_{\mathcal{A}}(w_{j-1},w_{j})+\tilde{d}_{\mathcal{A}}(w_{j},w_{j+1})\geq \tilde{d}_{\mathcal{A}}(w_{j-1},w_{j+1}).
\end{equation}
Indeed, if $j\neq j_{0}$, assuming by symmetry $j<j_{0}$, then by Lemma \ref{lem:Pro} and Proposition \ref{prop:Dist} the inequality (\ref{eqn:Three}) is an easy consequence of
\begin{multline*}
d\big(P_{\omega}[\psi_{j}](w_{j-1}),w_{j}\big)+d\big(P_{\omega}[\psi_{j+1}](w_{j}),w_{j+1}\big)\geq d\big(P_{\omega}[\psi_{j+1}](w_{j-1}),P_{\omega}[\psi_{j+1}](w_{j})\big)+\\
+ d\big(P_{\omega}[\psi_{j+1}](w_{j}),w_{j+1}\big)\geq d\big(P_{\omega}[\psi_{j+1}](w_{j-1}),w_{j+1}\big),
\end{multline*}
and of
\begin{multline*}
\sup_{a,b\in\mathcal{P}_{C_{j}}(X,\omega,\psi_{j-1})}\Big\{d(a,b)-d\big(P_{\omega}[\psi_{j}](a),P_{\omega}[\psi_{j}](b)\big)\Big\}+\\
+\sup_{a,b\in\mathcal{P}_{C_{j}}(X,\omega,\psi_{j})}\Big\{d(a,b)-d\big(P_{\omega}[\psi_{j+1}](a),P_{\omega}[\psi_{j+1}](b)\big)\Big\}\geq\\
\geq\sup_{a,b\in\mathcal{P}_{\max(C_{j-1},C_{j+1})}(X,\omega,\psi_{j-1})}\Big\{d(a,b)-d\big(P_{\omega}[\psi_{j+1}](a),P_{\omega}[\psi_{j+1}](b)\big)\Big\}.
\end{multline*}
In the case $j=j_{0}$, instead, assuming $\psi_{j-1}\preccurlyeq \psi_{j+1}$ the inequality (\ref{eqn:Three}) follows from
\begin{multline*}
d\big(P_{\omega}[\psi_{j}](w_{j-1}),w_{j}\big)+d\big(P_{\omega}[\psi_{j}](w_{j+1}),w_{j}\big)+\sup_{a,b\in \mathcal{P}_{C_{j}}(X,\omega,\psi_{j-1})}\Big\{d(a,b)-d\big(P_{\omega}[\psi_{j}](a),P_{\omega}[\psi_{j}](b)\big)\Big\}\geq\\
\geq d\big(P_{\omega}[\psi_{j}](w_{j-1}),P_{\omega}[\psi_{j}](w_{j+1})\big)+d\big(w_{j-1},P_{\omega}[\psi_{j-1}](w_{j+1})\big)-d\big(P_{\omega}[\psi_{j}](w_{j-1}),P_{\omega}[\psi_{j}](w_{j+1})\big)=\\
=d\big(w_{j-1},P_{\omega}[\psi_{j-1}](w_{j+1})\big).
\end{multline*}
Indeed it implies
\begin{multline*}
\tilde{d}_{\mathcal{A}}(w_{j-1},w_{j})+\tilde{d}_{\mathcal{A}}(w_{j},w_{j+1})\geq d\big(w_{j-1},P_{\omega}[\psi_{j-1}](w_{j+1})\big)+\\
+\sup_{a,b\in \mathcal{P}_{C_{j}}(X,\omega,\psi_{j+1})}\Big\{d(a,b)-d\big(P_{\omega}[\psi_{j}](a),P_{\omega}[\psi_{j}](b)\big)\Big\}+V_{w_{j+1}}-V_{w_{j-1}}\geq \tilde{d}_{\mathcal{A}}(w_{j-1},w_{j+1}).
\end{multline*}
Therefore, using again the inductive hypothesis, we may assume there exists $i_{0}\in\{0,\dots,m+2\}$ such that $C_{0}>C_{1}>\cdots>C_{i_{0}-1}\geq C_{i_{0}}$ and $C_{i_{0}}\leq C_{i_{0}+1}<\cdots<C_{m+1}<C_{m+2}$, where moreover $C_{i_{0}}<\max(C_{i_{0}-1},C_{i_{0}+1})$ (in the extreme cases $i_{0}=0,m+2$ the last inequality obviously restricts respectively to $C_{i_{0}}=C_{0}<C_{1}$ and to $C_{i_{0}}=C_{m+2}<C_{m+1}$).\\
\textbf{Step 4 ($\mathbf{m\to m+1}$): case $\mathbf{|i_{0}-j_{0}|>1}$.} By symmetry we may assume $i_{0}<j_{0}-1$. So $C_{j_{0}-2}\leq C_{j_{0}-1}<C_{j_{0}}<C_{j_{0}+1}$, which implies
$$
\tilde{d}_{\mathcal{A}}(w_{j_{0}-1},w_{j_{0}})+\tilde{d}_{\mathcal{A}}(w_{j_{0}},w_{j_{0}+1})\geq\tilde{d}_{\mathcal{A}}(w_{j_{0}-1},P_{\omega}[\psi_{j_{0}}](w_{j_{0}-1})\big)+\tilde{d}_{\mathcal{A}}\big(P_{\omega}[\psi_{j_{0}}](w_{j_{0}-1}),w_{j_{0}+1}\big)
$$
using the definition. Letting $\tilde{w}:=P_{\omega}[\psi_{j_{0}}](w_{j_{0}-1})$ and $\tilde{C}$ be the smallest non-negative real number such that $\tilde{w}\in \mathcal{P}_{\tilde{C}}(X,\omega,\psi_{j_{0}})$, we conclude this case by the argument exposed in the previous step since $\tilde{C}\leq C_{j_{0}-1}$ by construction and $C_{j_{0}-1}\geq C_{j_{0}-2}$.\\
\textbf{Step 5 ($\mathbf{m\to m+1}$): case $\mathbf{|i_{0}-j_{0}|=1}$} Let assume $i_{0}=j_{0}-1$. Since $C_{j_{0}-1}\leq C_{j_{0}}<C_{j_{0}+1}$, as in Step $4$, we can substitute $w_{j_{0}}$ by $P_{\omega}[\psi_{j_{0}}](w_{j_{0}-1})$. Therefore, up to replace $i_{0}$ by $i_{0}+1$, we have $i_{0}=j_{0}$ that is the last case addressed in the final step.\newline
\textbf{Step 6 ($\mathbf{m\to m+1}$): case $\mathbf{i_{0}=j_{0}}$} Since $C_{0}>C_{1}>\cdots>C_{j_{0}-1}>C_{j_{0}}$, alternating several times Proposition \ref{prop:Proper}$.(iv)$ and the triangle inequality for $d$ on $\mathcal{E}^{1}(X,\omega,\psi_{i})$ for $i=0,\dots,j_{0}-1$ we get
\begin{gather*}
\tilde{d}_{\mathcal{A}}(w_{0},w_{1})+\cdots+\tilde{d}_{\mathcal{A}}(w_{j_{0}-1},w_{j_{0}})\geq\\
\geq \tilde{d}_{\mathcal{A}}(w_{0},w_{1})+\cdots+\tilde{d}_{\mathcal{A}}(w_{j_{0}-2},w_{j_{0}-1})+d\big(w_{j_{0}-1},P_{\omega}[\psi_{j_{0}-1}](\varphi_{j_{0}})\big)\geq\\
\geq  \tilde{d}_{\mathcal{A}}(w_{0},w_{1})+\cdots+\tilde{d}_{\mathcal{A}}(w_{j_{0}-3},w_{j_{0}-2})+d\big(P_{\omega}[\psi_{j_{0}-1}](\varphi_{j_{0}-2}),P_{\omega}[\psi_{j_{0}-1}](\varphi_{j_{0}})\big)+\\
+\sup_{a,b\in \mathcal{P}_{C_{j_{0}-2}}(X,\omega,\psi_{j_{0}-2})}\Big\{d(a,b)-d\big(P_{\omega}[\psi_{j_{0}-1}](a),P_{\omega}[\psi_{j_{0}-1}](b)\big)\Big\}\geq\\
\geq \tilde{d}_{\mathcal{A}}(w_{0},w_{1})+\cdots+\tilde{d}_{\mathcal{A}}(w_{j_{0}-3},w_{j_{0}-2})+d\big(w_{j_{0}-2},P_{\omega}[\psi_{j_{0}-2}](\varphi_{j_{0}})\big)\geq\\
\geq\dots\geq \tilde{d}_{\mathcal{A}}(w_{0},w_{1})+d\big(w_{1},P_{\omega}[\psi_{1}](\varphi_{j_{0}})\big)\geq d\big(w_{0},P_{\omega}[\psi_{0}](\varphi_{j_{0}})\big).
\end{gather*}
Proceeding in the same way, by symmetry, we also get
$$
\tilde{d}_{\mathcal{A}}(w_{j_{0}},w_{j_{0}+1})+\cdots+\tilde{d}_{\mathcal{A}}(w_{m+1},w_{m+2})\geq d\big(P_{\omega}[\psi_{0}](\varphi_{j_{0}}),w_{m+2}\big).
$$
Hence
$$
\tilde{d}_{\mathcal{A}}(w_{0},w_{1})+\cdots+\tilde{d}_{\mathcal{A}}(w_{m+1},w_{m+2})\geq d(w_{0},w_{m+2})=d(u,v),
$$
which concludes the proof.
\subsection{Gromov-Hausdorff types of convergences \& direct limits: proof of Theorems \ref{thmD} and \ref{thmE}.}
\label{ssec:Consequences}
In this section we assume $\mathcal{A}=\{\psi_{k}\}_{k\in\mathbbm{N}}\subset \mathcal{M}^{+}$ to be a totally ordered subset such that $\psi_{k+1}\preccurlyeq \psi_{k}$ for any $k\in\mathbbm{N}$. Moreover we suppose that $\psi_{k}\searrow \psi\in \mathcal{M}^{+}$.
\begin{defn}
Let $\mathcal{A}$ and $\psi\in \mathcal{M}^{+}$ as above. Then the elements of the family
$$
\mathcal{K}_{\mathcal{A}}:=\bigcup_{k\in\mathbbm{N}}\Big\{K\subset \mathcal{E}^{1}(X,\omega,\psi) \, \mathrm{compact} \, \mathrm{such}\, \mathrm{that} \, K\subset P_{\omega}[\psi](\tilde{K})\, \mathrm{for} \,\tilde{K}\subset \mathcal{E}^{1}(X,\omega,\psi_{k}) \, \mathrm{compact} \Big\}
$$
are called \emph{$\mathcal{A}$-compact sets}.
\end{defn}
We recall that for a couple of compact metric spaces $ (X,d_{X}),(Y,d_{Y}) $, the Gromov-Hausdorff distance between them is defined as
$$
d_{GH}(X,Y)=\inf\{d^{d}_{H}(X,Y) \, :\, d \, \mathrm{admissible} \,\mathrm{distance}\, \mathrm{on} \, X\sqcup Y \}
$$
where a distance $d$ on $X\sqcup Y$ is said to be admissible if $d_{|X\times X}=d_{X}$ and $d_{|Y\times Y}=d_{Y}$ and where $d^{d}_{H}$ indicates the Hausdorff distance on the closed sets of $(X\sqcup Y, d)$.\\
A sequence of compact metric spaces $(X_{n},d_{n})$ converges in the Gromov-Hausdorff sense to a compact metric space $(X,d)$ if $d_{GH}(X_{n},X)\to 0$. We will use the notation $(X_{n},d_{n})\xrightarrow{GH} (X,d)$ and we refer to \cite{BBI01} and to \cite{BH99} for this notion of convergence.
\begin{prop}
\label{prop:GH}
For any $\mathcal{A}$-compact set $K\subset \mathcal{K}_{\mathcal{A}}$ there exists a sequence of strongly compact sets $(K_{k},d)\subset \big(\mathcal{E}^{1}(X,\omega,\psi_{k}),d\big)$ for $k\gg 1$ big enough such that
$$
(K_{k},d)\xrightarrow{GH} (K,d).
$$
\end{prop}
\begin{proof}
Let $k_{0}\in\mathbbm{N}$ such that $K\subset P_{\omega}[\psi](\tilde{K})$ for a compact set $\tilde{K}\subset \big(\mathcal{E}^{1}(X,\omega,\psi_{k_{0}}),d\big)$. Then we define
$$
K_{k_{0}}:=\tilde{K}\cap P_{\omega}[\psi]^{-1}(K),
$$
noting that it is a compact set in $\mathcal{E}^{1}(X,\omega,\psi_{k_{0}})$. Therefore we define for any $k\geq k_{0}$
$$
K_{k}:=P_{\omega}[\psi_{k}]\big(\tilde{K}\big)\cap P_{\omega}[\psi]^{-1}\big(K\big)=P_{\omega}[\psi_{k}]\big(K_{k_{0}}\big)
$$
and a correspondence $\mathcal{R}_{k}\subset K_{k}\times K$ as $(u_{k},u)\in \mathcal{R}_{k}$ if $u=P_{\omega}[\psi](u_{k})$. Thus to prove that $d_{GH}(K_{k},K)\to 0$ with respect to the $d$-distances it is enough to check that
$$
\mathrm{dis}\,\mathcal{R}_{k}:=\sup \big\{|d(u,v)-d(u_{k},v_{k})| \, :\, (u_{k},u),(v_{k},v)\in \mathcal{R}_{k}\big\}\to 0
$$
as $k\to \infty$ (see Theorem $7.3.25.$ in \cite{BBI01}). Hence Proposition \ref{prop:UC} concludes the proof.  
\end{proof}
For non-compact metric spaces there is a weaker notion of convergence than the Gromov-Hausdorff convergence, that is the \emph{pointed Gromov-Hausdorff} convergence. We recall that a sequence of pointed compact metric spaces $(K_{n},p_{n},d_{n})$ converges in the pointed Gromov-Hausdorff sense to $(K,p,d)$ if $d_{GH}\big((K_{n},p_{n}),(K,p)\big)\to 0$ as $n\to\infty$ where
$$
d_{GH}\big((K_{n},p_{n}),(K,p)\big):=\inf\big\{d^{d}_{H}(K_{n},K)+d(p_{n},p)\, :\, d \, \mathrm{admissible}\, \mathrm{metric}\, \mathrm{on} \, X\sqcup Y\big\}.
$$
Thus a sequence of non-compact pointed metric spaces $(X_{n},p_{n},d_{n})$ is said to converge in the pointed Gromov-Hausdorff sense to a non-compact pointed metric space $(X,p,d)$ if for any $r>0$
$$
d_{GH}\big((\overline{B_{r}(p_{n})},p_{n}),(\overline{B_{r}(p)},p)\big)\to 0
$$
as $n\to \infty$\footnote{This is actually not the right definition of point Gromov-Hausdorff convergence, but it is a characterization which holds when the sequence and the limit point are lenght spaces (\cite{BBI01}).}. We will use the notation $(X_{n},p_{n},d_{n})\xrightarrow{p-GH} (X,p,d)$.\\
If the pointed metric spaces are locally compact this convergence seems to be the most natural kind of convergence to look at. But if the pointed metric spaces are not locally compact, the pointed Gromov-Hausdorff convergence still seems a too strong kind of convergence. Thus we give the following general definition:
\begin{defn}
\label{defn:cp-GH}
A family of pointed metric spaces $(X_{n},p_{n},d_{n})$ converges in the \emph{compact pointed Gromov-Hausdorff convergence} to a pointed metric space $(X,p,d)$ if there exist a family of compact set $\{K_{j}\}_{j\in\mathbbm{N}}\subset X$ and, for any $n\in \mathbbm{N}$, a family of compact sets $\{K_{j,n}\}_{j\in\mathbbm{N}}\subset X_{n}$ such that
\begin{itemize}
\item[i)] $p_{n}\in K_{j,n}$ for any $n\in\mathbbm{N}$ and for any $j\in\mathbbm{N}$; \item[ii)] $p\in K_{j}$ for any $j\in\mathbbm{N}$;
\item[iii)] for any $n\in \mathbbm{N}$ fixed, $K_{j,n}\subset K_{j+1,n}$ for any $j\in\mathbbm{N}$ and $\bigcup_{j\in \mathbbm{N}}K_{j,n}$ is dense in $X_{n}$;
\item[iv)] $K_{j}\subset K_{j+1}$ for any $j\in\mathbbm{N}$ and $\bigcup_{j\in\mathbbm{N}}K_{j}$ is dense in $X$;
\item[v)] $d_{GH}\big((K_{j,n},p_{n}),(K_{j},p)\big)\to 0$.
\end{itemize}
We will use the notation $(X_{n},p_{n},d_{n})\xrightarrow{cp-GH} (X,p,d)$.
\end{defn}
We can now prove Theorem \ref{thmD}:
\begin{reptheorem}{thmD}
Let $\{\psi_{k}\}_{k\in\mathbbm{N}}\subset \mathcal{M}^{+}$ such that $\psi_{k}\searrow \psi\in\mathcal{M}^{+}$. Then
$$
\Big(\mathcal{E}^{1}(X,\omega,\psi_{k}),d\Big)\xrightarrow{cp-GH} \Big(\mathcal{E}^{1}(X,\omega,\psi),d\Big).
$$
\end{reptheorem}
\begin{proof}
For any $j\in\mathbbm{N}$ let $\mathcal{K}_{j}$ be the strongly compact set in $\mathcal{E}^{1}(X,\omega)$ containing all $\omega$-psh functions with bounded entropy by $j$ (see Theorem \ref{thm:Comp}). Thus, defining for any $j\in\mathbbm{N}$ and for any $k\in\mathbbm{N}$, $K_{j,k}:=P_{\omega}[\psi_{k}](\mathcal{K}_{j})$ and $K_{j}:=P_{\omega}[\psi](\mathcal{K}_{j})$, the theorem immediately follows from Lemma \ref{lem:ProDense} and Proposition \ref{prop:GH}.
\end{proof}

The maps $P_{k,j}:P_{\omega}[\psi_{j}](\cdot):\big(\mathcal{E}^{1}(X,\omega,\psi_{k}),d\big)\to\big(\mathcal{E}^{1}(X,\omega,\psi_{j}),d\big)$ for $k\leq j$ are morphisms in the category of metric spaces (see Lemma \ref{lem:Pro} and Proposition \ref{prop:Dist}). Moreover $\{P_{k,j}\}_{j\leq k, (k,j)\in\mathbbm{N}}$ produces a direct system again by Lemma \ref{lem:Pro}, and $\big\langle \big(\mathcal{E}^{1}(X,\omega,\psi),d\big), P_{k}\big\rangle$ is a \emph{target} of this direct system where $P_{k}:=P_{\omega}[\psi](\cdot):\big(\mathcal{E}^{1}(X,\omega,\psi_{k}),d\big)\to \big(\mathcal{E}^{1}(X,\omega,\psi),d\big)$.\\
We recall that a target $\big\langle(X,d_{X}),f_{X,n}\big\rangle$ of a direct system of metric spaces $\big\langle(X_{n},d_{n}), f_{n,m}\big\rangle$ is a metric space $(X,d_{X})$ with $1$-Lipschitz maps $f_{X,n}:(X_{n},d_{n})\to (X,d_{X})$ such that $f_{X,n}=f_{X,m}\circ f_{m,n}$ for any $n\leq m$.\\
Therefore since by the universal property the direct limit is the initial target, we immediately find out that the direct system $\big\langle \big(\mathcal{E}^{1}(X,\omega,\psi_{j}),P_{k,j}\big)$ admits a direct limit (recall that some direct systems in the category of metric spaces do not admit any not-trivial target like, for instance, the direct system $\big\langle(X_{n},d_{n}), f_{n,m}\big\rangle:=\big\langle(\mathbbm{R},\frac{1}{n}d_{eucl}), \mathrm{Id}\big\rangle$). We denote with $\mathfrak{m}-\lim_{\longrightarrow}$ the direct limit in the category of metric spaces.
\begin{reptheorem}{thmE}
There is an isometric embedding
$$
\mathfrak{m}-\lim_{\longrightarrow}\big\langle\big((\mathcal{E}^{1}(X,\omega,\psi_{i}),d), P_{i,j}\big)\big\rangle\hookrightarrow \big(\mathcal{E}^{1}(X,\omega,\psi),d\big)
$$
with dense image. More precisely the direct limit is isometric to $\Big(\bigcup_{k\in\mathbbm{N}}P_{\omega}[\psi]\big(\mathcal{E}^{1}(X,\omega,\psi_{k})\big),d\Big)$.
\end{reptheorem}
\begin{proof}
As a consequence of Lemma \ref{lem:ProDense} the metric subspace $T:=\bigcup_{k\in\mathbbm{N}}P_{\omega}[\psi]\big(\mathcal{E}^{1}(X,\omega,\psi_{k})\big)$ is dense in $\big(\mathcal{E}^{1}(X,\omega,\psi),d\big)$. Then, since as stated before $\langle (T,d),P_{k})$ is a target of the direct system considered, to conclude the proof it is enough to show that for any other target $\langle(Y,d_{Y}),P_{Y,k}\big\rangle$ there exists a $1-$Lipschitz map $P_{Y,T}:T\to Y$ such that $P_{Y,T}\circ P_{k}=P_{Y,k}$ for any $k\in\mathbbm{N}$.\\
Therefore, letting $\big\langle (Y,d_{Y}),P_{Y,k}\big\rangle$ a target, for any $u\in T$ we denote with $k_{u}\in\mathbbm{N}$ the minimum natural number $k$ such that $u\in P_{\omega}[\psi]\big(\mathcal{E}^{1}(X,\omega,\psi_{k})\big)$ and we fix a function $\varphi_{u}\in \mathcal{E}^{1}(X,\omega,\psi_{k_{u}})$ such that $P_{\omega}[\psi](\varphi_{u})=u$. Next we define $P_{Y,T}:T\to Y$ as
$$
P_{Y,T}(u):=P_{Y,k_{u}}(\varphi_{u}),
$$
i.e. it is defined so that $P_{Y,T}\circ P_{k}=P_{Y,k}$ for any $k\in\mathbbm{N}$. Note that the definition does not depend on representatives since $P_{Y,k_{1}}(\varphi_{1})=P_{Y,k_{2}}(\varphi_{2})$ for $\varphi_{1}\in\mathcal{E}^{1}(X,\omega,\psi_{k_{1}}),\varphi_{2}\in\mathcal{E}^{1}(X,\omega,\psi_{k_{2}})$ if $P_{k_{1}}(\varphi_{1})=P_{k_{2}}(\varphi_{2})$. Indeed
\begin{multline*}
d_{Y}\big(P_{Y,k_{1}}(\varphi_{1}),P_{Y,k_{2}}(\varphi_{2})\big)=d_{Y}\big(P_{Y,j}\circ P_{j,k_{1}}(\varphi_{1}),P_{Y,j}\circ P_{j,k_{2}}(\varphi_{2})\big)\leq\\
\leq d\big(P_{j,k_{1}}(\varphi_{1}),P_{j,k_{2}}(\varphi_{2})\big)\to d(P_{k_{1}}(\varphi_{1}),P_{k_{2}}(\varphi_{2}))=0
\end{multline*}
as $j\to \infty$ by Proposition \ref{prop:UC}.\\
To finish the proof it remains to check that $P_{Y,T}$ is $1$-Lipschitz. Fixed $u,v\in T$, we have for any $j\in\mathbbm{N}$ big enough
$$
d_{Y}\big(P_{Y,T}(u),P_{Y,T}(v)\big)=d_{Y}\big(P_{Y,j}\circ P_{j,k_{u}}(\varphi_{u}),P_{Y,j}\circ P_{j,k_{v}}(\varphi_{v})\big)\leq d(P_{j,k_{u}}(\varphi_{u}),P_{j,k_{v}}(\varphi_{v})),
$$
where $P_{k_{u}}(\varphi_{u})=u$, $P_{k_{v}}(\varphi_{v})=v$. Hence $d_{Y}\big(P_{Y,T}(u),P_{Y,T}(v)\big)\leq d(u,v)$ letting $j\to +\infty$.
\end{proof}

\end{document}